\documentclass[reqno,10pt,centertags,draft]{amsart}
\usepackage{amsmath,amsthm,amscd,amssymb,latexsym,upref,enumerate}
\usepackage{bbm}
\usepackage{nicefrac}
\usepackage{hyperref} 
\newcommand*{\mailto}[1]{\href{mailto:#1}{\nolinkurl{#1}}}
\newcommand{\arxiv}[1]{\href{http://arxiv.org/abs/#1}{arXiv:#1}}

\DeclareRobustCommand\dash{%
\unskip\nobreak\thinspace\textemdash\thinspace\ignorespaces}
\newcommand{\R}{{\mathbb R}}

\newcommand{\C}{{\mathbb C}}

\newcommand{\bbC}{{\mathbb{C}}}

\newcommand{\bbN}{{\mathbb{N}}}

\newcommand{\bbR}{{\mathbb{R}}}
\newcommand{\bbT}{{\mathbb{T}}}

\newcommand{\cH}{{\mathcal H}}

\DeclareMathOperator{\ran}{ran}
\DeclareMathOperator{\dom}{dom}

\DeclareMathOperator*{\sgn}{sgn}

\renewcommand{\Re}{\text{\rm Re}}

\newcommand{\no}{\notag}
\newcommand{\lb}{\label}
\newcommand{\f}{\frac}

\newcommand{\ol}{\overline}

\newcommand{\wti}{\widetilde}

\newcommand{\hatt}{\widehat} 
\newcommand{\dott}{\,\cdot\,}

\renewcommand{\dot}{\overset{\textbf{\Large.}}}

\newcommand{\bi}{\bibitem}

\let\geq\geqslant
\let\leq\leqslant

\newcommand{\la}{\lambda}

\newcommand{\Lr}{{L^2((a,b);r dx)}}

\makeatletter
\def\theequation{\@arabic\c@equation}
\allowdisplaybreaks 
\numberwithin{equation}{section}

\newtheorem{theorem}{Theorem}[section]

\newtheorem{lemma}[theorem]{Lemma}
\newtheorem{corollary}[theorem]{Corollary}
\newtheorem{definition}[theorem]{Definition}
\newtheorem{hypothesis}[theorem]{Hypothesis}

\theoremstyle{remark}
\newtheorem{remark}[theorem]{Remark}

\begin{document}

\title[Strict Domain Monotonicity and Lower Bounds]{Strict Domain Monotonicity of the Principal Eigenvalue and a Characterization of Lower Boundedness for the Friedrichs Extension of Four-Coefficient Sturm--Liouville Operators}

\author[F.\ Gesztesy]{Fritz Gesztesy}
\address{Department of Mathematics, 
Baylor University, Sid Richardson Bldg., 1410 S.\,4th Street, Waco, TX 76706, USA}
\email{\mailto{Fritz\_Gesztesy@baylor.edu}}
%\email{Fritz$\_$Gesztesy@baylor.edu}
\urladdr{\url{http://www.baylor.edu/math/index.php?id=935340}}
%\urladdr{http://www.baylor.edu/math/index.php?id=935340}

\author[R.\ Nichols]{Roger Nichols}
\address{Department of Mathematics (Dept.~6956), The University of Tennessee at Chattanooga, 
615 McCallie Avenue, Chattanooga, TN 37403, USA}
\email{\mailto{Roger-Nichols@utc.edu}}
%\email{Roger-Nichols@utc.edu}
\urladdr{\url{https://sites.google.com/mocs.utc.edu/rogernicholshomepage/home}}
%\urladdr{http://www.utc.edu/faculty/roger-nichols/index.php}

\date{\today}
%\thanks{} 
%\thanks{Appeared in {\it .}
\@namedef{subjclassname@2020}{\textup{2020} Mathematics Subject Classification}

\dedicatory{Dedicated to the memory of Roger T.\ Lewis (1942--2021)}

%%%%%%%%%%%%%%%%%%%%%%%%%%%%%%%
\begin{abstract} 
Using the variational characterization of the principal (i.e., smallest) eigenvalue below the essential spectrum of a lower semibounded self-adjoint operator, we prove strict domain monotonicity (with respect to changing the finite interval length) of the principal eigenvalue of the Friedrichs extension $T_F$ of the minimal operator for regular four-coefficient Sturm--Liouville differential expressions. In the more general singular context, these four-coefficient differential expressions act according to
\[
 \tau f = \frac{1}{r} \left( - \big(f^{[1]}\big)' + s f^{[1]} + qf\right)\,\text{ with $f^{[1]} = p [f' + s f]$ on $(a,b) \subseteq \bbR$},
\] 
where  the coefficients $p$, $q$, $r$, $s$ are real-valued and Lebesgue measurable on $(a,b)$, with $p > 0$, $r>0$ a.e.\ on $(a,b)$, and $p^{-1}$, $q$, $r$, $s \in L^1_{loc}((a,b); dx)$,  
and $f$ is supposed to satisfy
\[
f \in AC_{loc}((a,b)), \; p[f' + s f] \in AC_{loc}((a,b)).  
\]
This setup is sufficiently general so that $\tau$ permits certain distributional potential coefficients $q$, including potentials in $H^{-1}_{loc}((a,b))$. 

As a consequence of the strict domain monotonicity of the principal eigenvalue of the Friedrichs extension in the regular case, and on the basis of oscillation theory in the singular context, in our main result, we characterize all lower bounds of $T_F$ as those $\lambda\in \bbR$ for which the differential equation $\tau u = \lambda u$ has a strictly positive solution $u > 0$ on $(a,b)$.
\end{abstract}
%%%%%%%%%%%%%%%%%%%%%%%%%%%%%%%

\maketitle

{\scriptsize{\tableofcontents}}
%\normalsize

\maketitle

%\newpage 

%{\scriptsize{\tableofcontents}}
%\normalsize

%%%%%%%%%%%%%%%%%%%%%%%%%%%%%%%
%%%%%%%%%%%%%%%%%%%%%%%%%%%%%%%
\section{Introduction} \lb{s1}
%%%%%%%%%%%%%%%%%%%%%%%%%%%%%%%
%%%%%%%%%%%%%%%%%%%%%%%%%%%%%%%

To set the stage and at the same time motivate our interest in singular Sturm--Liouville operators on an arbitrary interval $(a,b) \subseteq \bbR$ associated with general four-coefficient differential expressions of the type 
\begin{equation}
 \tau f = \frac{1}{r} \left( - \big(p[f' + s f]\big)' + s p[f' + s f] + qf\right),    \lb{1.1}
\end{equation}
following \cite{EGNT13}, one assumes the coefficients 
\begin{align} 
\begin{split} 
& p, q, r, s \, \text{ to be real-valued and Lebesgue measurable on $(a,b)$},    \\
& p, r>0 \, \text{ a.e.\ on $(a,b)$,} \quad p^{-1}, q, r, s \in L^1_{loc}((a,b); dx).    \lb{1.2} 
\end{split} 
\end{align}
Here $f$ is supposed to satisfy
\begin{equation} 
f \in AC_{\text{loc}}((a,b)), \; p[f' + s f] \in AC_{\text{loc}}((a,b)), 
\end{equation}
with $AC_{loc}((a,b))$ denoting the set of locally absolutely continuous functions on $(a,b)$. The expression 
\begin{equation} 
f^{[1]} = p[f'+s f] 
\end{equation}
represents the {\it first quasi-derivative} of $f$.

One notes that in the general case \eqref{1.1}, the differential expression is formally given by
\begin{equation}
\tau f = \frac{1}{r} \left( - \big(pf'\big)' + \big[- (p s)' + p s^2 + q\big]f \right).
\end{equation}
Moreover, in the special case $s\equiv 0$ this approach reduces to the standard one, that is, 
one obtains,
\begin{equation}
\tau f = \frac{1}{r} \left( - \big(pf'\big)' + q f \right).
\end{equation}
 
In the particular case $p=r=1$ this approach is sufficiently general to include arbitrary distributional potential coefficients from $H^{-1}_{loc}((a,b)) = W^{-1,2}_{loc}((a,b))$ (as the term 
$s^2$ can be absorbed in $q$). Moreover, since there are distributions in $H^{-1}_{loc}((a,b))$ which are not measures, the operators discussed here are not a special case of Sturm--Liouville operators with measure-valued coefficients as discussed, for instance, in \cite{ET12}.

We emphasize that similar differential expressions have already been studied by 
Bennewitz and Everitt \cite{BE83} in 1983 (see also \cite[Sect.\ I.2]{EM99}). An extremely thorough and systematic investigation, including even and odd higher-order operators defined in terms of appropriate quasi-derivatives, and in the general case of matrix-valued coefficients (including distributional potential coefficients in the context of Schr\"odinger-type operators) was presented by Weidmann \cite{We87} in 1987. In fact, the general approach in \cite{BE83} and \cite{We87} draws on earlier discussions of quasi-derivatives in Shin \cite{Sh38}--\cite{Sh43}, 
Naimark \cite[Ch.\ V]{Na68}, and Zettl \cite{Ze75}. There were also earlier papers dealing with Schr\"odinger operators involving strongly singular and oscillating potentials which should be mentioned in this context, such as, Baeteman and Chadan \cite{BC75}, \cite{BC76}, Combescure \cite{Co80}, 
Combescure and Ginibre \cite{CG76}, Herczy\'nski \cite{He89}, Pearson \cite{Pe79}, and 
Rofe-Beketov and Hristov \cite{RH66}, \cite{RH69}. 

In addition, the case of point interactions as particular distributional potential coefficients in Schr\"odinger operators received enormous attention, too numerous to be mentioned here in detail. Hence, we only refer to the standard monographs by Albeverio, Gesztesy, H{\o}egh-Krohn, and Holden \cite{AGHKH05} and  
Albeverio and Kurasov \cite{AK01}, and some of the more recent developments in 
Albeverio, Kostenko, and Malamud \cite{AKM10}, and Kostenko and Malamud 
\cite{KM10}--\cite{KM13}.

In 1999 Savchuk and Shkalikov \cite{SS99} started a new development for Sturm--Liouville (resp., Schr\"odinger) operators with distributional potential coefficients
in connection with 
areas such as self-adjointness proofs, spectral and inverse spectral theory, oscillation properties, spectral properties in the non-self-adjoint context, etc. In addition to the important series of papers by Savchuk and Shkalikov \cite{SS99}--\cite{SS10}, we also mention other groups such as Albeverio, Hryniv, and Mykytyuk \cite{AHM08}, Bak and Shkalikov \cite{BS02},  Ben Amara and Shkalikov \cite{BS09}, Ben Amor and Remling \cite{BR05}, Davies \cite{Da13}, 
Djakov and Mityagin \cite{DM09}--\cite{DM12}, Eckhardt and Teschl \cite{ET12}, 
Frayer, Hryniv, Mykytyuk, and Perry \cite{FHMP09}, Gesztesy and Weikard \cite{GW13}, 
Goriunov and Mikhailets \cite{GM10}, \cite{GM10a}, Goriunov, Mikhailets, and Pankrashkin \cite{GMP13}, Guliyev \cite{Gu19}, Hryniv \cite{Hr11}, 
Kappeler and M\"ohr \cite{KM01}, Kappeler, Perry, Shubin, and Topalov \cite{KPST05}, 
Kappeler and Topalov \cite{KT04}, 
Hryniv and Mykytyuk \cite{HM01}--\cite{HM12}, Hryniv, Mykytyuk, and Perry \cite{HMP11}--\cite{HMP11a}, 
Kato \cite{Ka10}, Korotyaev \cite{Ko03}, \cite{Ko12}, Maz'ya and Shaposhnikova \cite[Ch.\ 11]{MS09}, 
Maz'ya and Verbitsky \cite{MV02}--\cite{MV06}, Mikhailets and Molyboga \cite{MM04}--\cite{MM09}, 
Mirzoev and Safanova \cite{MS11}, Mykytyuk and Trush \cite{MT10}, 
and Sadovnichaya \cite{Sa10}, \cite{Sa11}. For a modern treatment we also refer to the very general contribution by Ghatasheh and Weikard \cite{GW20}. 

Finally, we also mention that some of the attraction in connection with distributional potential coefficients 
in the Schr\"odinger operator clearly stems from the low-regularity investigations of solutions of the 
Korteweg--de Vries (KdV) equation. We mention, for instance, Buckmaster and Koch \cite{BK15}, 
Grudsky and Rybkin \cite{GR14}, Kappeler and M\"ohr \cite{KM01}, Kappeler and Topalov \cite{KT05}, \cite{KT06}, and Rybkin \cite{Ry10}. 

Returning to the topic at hand, we discuss the Friedrichs (or Dirichlet) extension $T_F$ of the minimal operator associated with $\tau$ in the regular and singular context (see Definition \ref{d2.2} and \eqref{2.12}) and as our principal results will prove the following two facts recorded in Theorems \ref{t1.1} and \ref{t1.2} below. \\[1mm] 
Let $T_{F,(\alpha,\beta)}$ be the Friedrichs extension of $T_{min,(\alpha,\beta)}$ associated with 
$\tau|_{(\alpha,\beta)}$, where $(\alpha,\beta)\subseteq (a,b)$. In addition, denote by  
$\lambda_1(A)$ the lowest eigenvalue of the self-adjoint operator $A$ in a Hilbert space $\cH$ (implicitly assuming the discrete spectrum below the infimum of the essential spectrum of $A$ is nonempty). Then the strict domain monotonicity of the principal (i.e., smallest) eigenvalue of 
$T_F$ proved in Theorem \ref{t3.4} can be summarized as follows: 
%%%%%%%
\begin{theorem} \lb{t1.1}
Suppose $\tau$ is regular on the finite interval $(a,b)$ $($see Definition \ref{d2.2}$)$. Let 
$T_{F,(c,d)}$ denote the Friedrichs extension of $T_{min,(c,d)}$ associated with $\tau|_{(c,d)}$, where $(c,d)\subset (a,b)$ and $c<d$; that is, either $a\leq c$ and $d<b$, or else $a<c$ and 
$d\leq b$, so that $(c,d)$ is strictly contained in $(a,b)$. Then
\begin{equation}
\lambda_1(T_{F,(a,b)}) < \lambda_1(T_{F,(c,d)}).
\end{equation}
 \end{theorem}
 %%%%%%%

Our second principal result, proved in Theorem \ref{t4.3}, characterizes all lower bounds of 
$T_{F,(a,b)}$ as those $\lambda\in \bbR$ for which the differential equation $\tau u = \lambda u$ has a strictly positive solution $u > 0$ on $(a,b)$. More precisely, we will prove the following facts:
%%%%%%%%%%%
\begin{theorem} \lb{t1.2} 
Assume \eqref{1.2} and let $\lambda_0\in \bbR$.  Then the following statements $(i)$--$(iv)$ are equivalent. \\[1mm]
$(i)$  $T_{min,(a,b)}\geq \lambda_0I_{L^2((a,b);r\,dx)}$.\\[1mm]
$(ii)$ $T_{F,(a,b)}\geq \lambda_0I_{L^2((a,b);r\,dx)}$.\\[1mm]
$(iii)$  $\tau-\lambda_0$ is disconjugate $($cf.\ Definition \ref{d2.16}$)$ on $(a,b)$.\\[1mm]
$(iv)$  There exists a strictly positive solution $u_0(\lambda_0,\,\cdot\,)>0$ of 
$(\tau - \lambda_0)u = 0$ on $(a,b)$.
\end{theorem}
%%%%%%%%%%%

The necessary background for four-coefficient Sturm--Liouville operators is provided in Section \ref{s2}; Sections \ref{s3} and \ref{s4} contain our principal results, including some variational characterizations of eigenvalues and eigenfunctions. 
 
Finally, a few remarks on the notation employed: Given a separable complex Hilbert space $\cH$, $(\dott,\dott)_{\cH}$ denotes the scalar product in $\cH$ (linear in the second argument), and $I_{\cH}$ represents the identity operator in $\cH$. The domain and range of a linear operator $T$ in $\cH$ are abbreviated by $\dom(T)$ and $\ran(T)$. The closure of a closable operator $S$ is denoted by $\ol S$. The kernel (null space) of $T$ is denoted by
$\ker(T)$. The spectrum, discrete spectrum, and essential spectrum of a closed linear operator in 
$\cH$ will be abbreviated by $\sigma(\cdot)$, $\sigma_{d}(\cdot)$, and $\sigma_{ess}(\cdot)$,  respectively.

%%%%%%%%%%%%%%%%%%%%%%%%%%%%%%%
%%%%%%%%%%%%%%%%%%%%%%%%%%%%%%%
\section{Four-Coefficient Sturm--Liouville Operators} \lb{s2}
%%%%%%%%%%%%%%%%%%%%%%%%%%%%%%%
%%%%%%%%%%%%%%%%%%%%%%%%%%%%%%%

In this section we summarize the basics of four-coefficient Sturm--Liouville differential expressions and their associated Hilbert space operators to be used in the sequel.  In particular, since the main aim of Section \ref{s4} is to characterize the lower bounds of the Friedrichs extension of the underlying minimal operator, we recall several results that connect lower semiboundedness of the minimal operator with non-oscillation properties of the underlying differential expression.  These results lead to a characterization of the Friedrichs extension in terms of principal/nonprincipal solutions and generalized boundary values entirely analogous to those introduced for three-coefficient expressions in \cite{GLN20}.  With the exception of Theorems \ref{t2.18}
 and \ref{t2.20}
 and part of Corollary \ref{c2.24}, all of the results in this section may be found with complete proofs in \cite{EGNT13}.  Although Theorems \ref{t2.18}
 and \ref{t2.20}
 are new in the context of four-coefficient Sturm--Liouville expressions, the proof of Theorem \ref{t2.18}
 is identical to the analogous result for three-coefficient expressions obtained in \cite[Theorem 4.5]{GLN20}, and Theorem \ref{t2.20}
 then follows as a restatement of \cite[Theorem 6.4]{EGNT13} in terms of generalized boundary values.

We begin by introducing a basic hypothesis which is assumed throughout this paper.

%%%%%%%%%%%%
\begin{hypothesis} \lb{h2.1}
Let $-\infty\leq a<b\leq \infty$.  Suppose that $p$, $q$, $r$, $s$ are Lebesgue measurable on $(a,b)$ with $p^{-1}, q, r, s\in L^1_{loc}((a,b); dx)$ and real-valued a.e.~on $(a,b)$ with $r>0$ and $p>0$  a.e.~on $(a,b)$.
\end{hypothesis}
%%%%%%%%%%%%

Assuming Hypothesis \ref{h2.1}, we introduce the set
\begin{equation}
\mathfrak{D}_{\tau}((a,b))=\big\{g\in AC_{loc}((a,b))\, \big|\, g^{[1]}=p [g' + s g] \in AC_{loc}((a,b))\big\}
\end{equation}
and the differential expression $\tau$ defined by
\begin{equation}
 \tau f = \frac{1}{r} \Big( - \big(f^{[1]}\big)' + s f^{[1]} + qf\Big) \in L_{loc}^1((a,b);r\,dx),\quad f\in \mathfrak{D}_{\tau}((a,b)),  \lb{2.2}
\end{equation}
where the expression
\begin{equation}
f^{[1]}=p [f' + s f], \quad f \in \mathfrak{D}_{\tau}((a,b)),    \lb{2.3}
\end{equation}
is the \textit{first quasi-derivative} of $f$.  For each $f,g\in \mathfrak{D}_{\tau}((a,b))$, the (modified) Wronskian of $f$ and $g$ is defined by
\begin{equation}\lb{2.4}
W(f,g)(x) = f(x)g^{[1]}(x) - f^{[1]}(x)g(x),\quad x\in (a,b).
\end{equation}
One infers that $W(f,g)$ is locally absolutely continuous on $(a,b)$ and its derivative is
\begin{equation}
W(f,g)'(x) = \big[g(x) (\tau f)(x) - f(x) (\tau g)(x)\big] r(x), \quad x\in(a,b). 
\end{equation}
In particular, if $z\in \bbC$, then the Wronskian of two solutions $u_j(z,\,\cdot\,)\in \mathfrak{D}_{\tau}((a,b))$, $j\in\{1,2\}$, of $(\tau - z)u = 0$ on $(a,b)$ is constant.  Moreover, $W(u_1(\lambda,\,\cdot\,),u_2(\lambda,\,\cdot\,))\neq 0$ if and only if $u_1(\lambda,\,\cdot\,)$ and $u_2(\lambda,\,\cdot\,)$ are linearly independent.

%%%%%%%%%%%
\begin{definition} \lb{d2.2} 
The differential expression $\tau$ is said to be \textit{regular} on $(a,b)$ if $-\infty<a<b<\infty$ $($i.e., $a$ and $b$ are finite$)$ and $p^{-1}, q, r, s\in L^1((a,b); dx)$; otherwise, $\tau$ is said to be \textit{singular} on $(a,b)$.
\end{definition}
%%%%%%%%%%%

If $\tau$ is regular on $(a,b)$, then for each $f\in \mathfrak{D}_{\tau}((a,b))$ the following limits exist and are finite:
\begin{equation}\lb{2.6a}
\begin{split}
&f(a):=\lim_{x\downarrow a}f(x),\quad f^{[1]}(a):=\lim_{x\downarrow a}f^{[1]}(x),\\
&f(b):=\lim_{x\uparrow b}f(x),\quad f^{[1]}(b):=\lim_{x\uparrow b}f^{[1]}(x).
\end{split}
\end{equation}

We recall the following existence and uniqueness result for initial-value problems corresponding to the differential equation $(\tau-z)f=g$.
%%%%%%%%%%%
\begin{theorem}[{\cite[Theorem 2.2]{EGNT13}}]\label{t2.3}

If $g\in L^1_{\rm{loc}}((a,b);r\,dx)$, $z\in\C$, $c\in(a,b)$, and $\alpha$, $\beta\in\C$, then there is a unique solution  $f\in\mathfrak{D}_{\tau}((a,b))$ satisfying $(\tau - z) f = g$, $f(c)=\alpha$, and $f^{[1]}(c)=\beta$.  If, in addition, $g$, $\alpha$, $\beta$, and $z$ are real-valued, then the solution $f$ is real-valued.
\end{theorem}
%%%%%%%%%%%

%%%%%%%%%%%
\begin{definition} \lb{d2.4} 
A measurable function $f$ lies in $\Lr$ near $a$ $($resp., near $b$$)$ if $f$ lies in $L^2((a,c);r\, dx)$ $($resp., in $L^2((c,b);r\, dx)$$)$ for each $c\in(a,b)$.
\end{definition}
%%%%%%%%%%%

The celebrated Weyl alternative may then be stated as follows.

%%%%%%%%%%%%%
\begin{theorem}[Weyl's Alternative, {\cite[Section 4]{EGNT13}}] \lb{t2.5}
 ${}$ \\
Assume Hypothesis \ref{h2.1}. Then the following alternative holds.  Either: \\[1mm] 
$(i)$ For every $z\in\bbC$, all solutions $u$ of $(\tau-z)u=0$ are in $L^2((a,b);r dx)$ near $b$ 
$($resp., near $a$$)$, \\[1mm] 
or, \\[1mm] 
$(ii)$ For every $z\in\bbC$, there exists at least one solution $u$ of $(\tau-z)u=0$ which is not in $L^2((a,b);r dx)$ near $b$ $($resp., near $a$$)$. In this case, for each $z\in\bbC\backslash\bbR$, there exists precisely one solution $u_b$ $($resp., $u_a$$)$ of $(\tau-z)u=0$ $($up to constant multiples$)$ which lies in $L^2((a,b);r dx)$ near $b$ $($resp., near $a$$)$. 
\end{theorem}
%%%%%%%%%%

Weyl's alternative yields the following limit circle/limit point classification of $\tau$ at an interval endpoint.

%%%%%%%%%%%%%
\begin{definition} \lb{d2.6} 
Assume Hypothesis \ref{h2.1}. \\[1mm]  
In case $(i)$ in Theorem \ref{t2.5}, $\tau$ is said to be in the \textit{limit circle case} at $b$ $($resp., at $a$$)$. \\[1mm] 
In case $(ii)$ in Theorem \ref{t2.5}, $\tau$ is said to be in the \textit{limit point case} at $b$ $($resp., at $a$$)$. \\[1mm]
If $\tau$ is in the limit circle case at $a$ and $b$ then $\tau$ is also called \textit{quasi-regular} on $(a,b)$. 
\end{definition}
%%%%%%%%%%%%%

The differential expression $\tau$ gives rise to linear operators in the Hilbert space $L^2((a,b);r\,dx)$ equipped with the inner product
\begin{equation}
(f,g)_{L^2((a,b); r\,dx)} = \int_a^b r(x)\, dx\, \overline{f(x)}g(x),\quad f,g\in L^2((a,b);r\,dx).
\end{equation}
The \textit{maximal operator} associated to $\tau$ is denoted by $T_{max}$ and is defined by
\begin{align}
&T_{max}f = \tau f,\\
&f\in \dom(T_{max})=\big\{g\in L^2((a,b);r\,dx)\,\big|\, g\in \mathfrak{D}_{\tau}((a,b)),\, \tau g\in L^2((a,b);r\,dx)\big\}.\no
\end{align}
The Wronskian of any two functions $f,g\in \dom(T_{max})$ possesses finite boundary values at the endpoints of $(a,b)$; that is, the following limits exist and are finite:
\begin{equation}\lb{2.9a}
W(f,g)(a) := \lim_{x\downarrow a}W(f,g)(x),\quad W(f,g)(b) := \lim_{x\uparrow b}W(f,g)(x).
\end{equation}
In particular,
\begin{equation}\lb{2.10a}
\text{$W(f,g)$ is bounded on $(a,b)$ for every $f,g\in \dom(T_{max})$.}
\end{equation}
The \textit{pre-minimal operator} associated to $\tau$ is denoted by $\dot T$ and is defined by
\begin{align}
& \dot T f = \tau f,\lb{2.10}\\
& f \in \dom\big(\dot T\big) = \big\{ g\in\dom(T_{max}) \,|\, g \text{ has compact support in }(a,b)\big\}.\no
\end{align}
By \cite[Theorem 3.4]{EGNT13}, $\dot T$ is densely defined and symmetric in $L^2((a,b);r\,dx)$ with $\big(\dot T\big)^*=T_{max}$.  The \textit{minimal operator} associated to $\tau$ is denoted by $T_{min}$ and is defined to be the closure of the pre-minimal operator:
\begin{equation}
T_{min} := \overline{\dot T}.     \lb{2.12}
\end{equation}
In addition, $T_{min}$ and $T_{max}$ are adjoint to one another:
\begin{equation}\lb{2.9}
T_{min}^*=T_{max}\quad \text{and}\quad T_{max}^*=T_{min},
\end{equation}
so that $T_{min}$ is a closed symmetric operator.

The minimal operator $T_{min}$ has deficiency indices $(d,d)$ with $d\in \{0,1,2\}$; that is,
\begin{equation}
\dim \big(\ker(T_{max}\pm I_{L^2((a,b); r\,dx)})\big) = d,
\end{equation}
where $d$ is the number of limit circle endpoints for $\tau$.  In particular, $T_{min}$ has self-adjoint extensions.  The adjoint relations in \eqref{2.9} imply that a self-adjoint operator $T$ in $L^2((a,b);r\,dx)$ is an extension of $T_{min}$ if and only if $T$ is a restriction of $T_{max}$.  Therefore, every self-adjoint extension of $T_{min}$ is uniquely characterized by its domain.  Parametrizations of the self-adjoint extensions of $T_{min}$ are given in terms of Wronskians and linear functionals in \cite[Section 5]{EGNT13} and \cite[Section 6]{EGNT13}, respectively.  Below in Theorem \ref{t2.20}
 we shall present an alternative characterization in terms of principal/nonprincipal solutions and generalized boundary values when $T_{min}$ is lower semibounded.

One recalls that a symmetric operator $S$ in a Hilbert space $(\cH,(\,\cdot\,,\,\cdot\,)_{\cH})$ is \textit{lower semibounded} (or \textit{bounded from below}) if there exists $\lambda_0\in \bbR$ such that
\begin{equation}
(u,Su)_{\cH}\geq \lambda_0(u,u)_{\cH},\quad u\in \dom(S).
\end{equation}
In this case, $\lambda_0$ is called a \textit{lower bound for} $S$, we write $S\geq \lambda_0 I_{\cH}$, and the \textit{greatest lower bound} of $S$ is defined to be
\begin{equation}
\lambda_0(S):= \inf_{0\neq u\in \dom(S)} \frac{(u,Su)_{\cH}}{(u,u)_{\cH}}.
\end{equation}
The existence of principal and nonprincipal solutions is closely connected to oscillation theory for $\tau-\lambda$.  Therefore, we recall the following definition.

%%%%%%%%%%%
\begin{definition} \lb{d2.7} 
Suppose Hypothesis \ref{h2.1} holds and let $\lambda \in \bbR$.  The differential expression $\tau-\lambda$ is called {\it oscillatory at $a$} $($resp., $b$$)$ if some solution of $(\tau - \lambda) u = 0$ has infinitely many zeros accumulating at $a$ $($resp., $b$$)$; otherwise, $\tau-\lambda$ is called {\it nonoscillatory at $a$} $($resp., $b$$)$.
\end{definition}
%%%%%%%%%%%

%%%%%%%%%%%
\begin{theorem}[{\cite[Theorem 11.4]{EGNT13}}] \lb{t2.8}

Assume Hypothesis \ref{h2.1} and let $\lambda \in \bbR$ be fixed.  If $\tau-\lambda$ is nonoscillatory at $b$, then there exists a real-valued solution $u_b(\lambda,\dott)$ of $(\tau - \lambda) u = 0$ satisfying the following properties $(i)$--$(iii)$ in which $\hatt u_b(\lambda,\dott)$ denotes an arbitrary real-valued solution of $(\tau - \lambda) u = 0$ linearly independent of $u_b(\lambda,\dott)$.\\[1mm]
$(i)$ $u_b(\lambda,\dott)$ and $\hatt u_b(\lambda,\dott)$ satisfy the limiting relation
\begin{equation}\lb{2.15a}
\lim_{x\uparrow b}\frac{u_b(\la,x)}{\hatt u_b(\la,x)}=0.
\end{equation}
$(ii)$  $u_b(\lambda,\dott)$ and $\hatt u_b(\lambda,\dott)$ satisfy
\begin{equation}\lb{2.16a}
\int ^b dx \, |p(x)|^{-1}\hatt u_b(\la,x)^{-2}<\infty \, \text{ and } \, 
\int ^b dx \, |p(x)|^{-1}u_b(\la,x)^{-2}=\infty.   
\end{equation}
$(iii)$ Suppose $x_0\in(a,b)$ strictly exceeds the largest zero, if any, of $u_b(\lambda,\dott)$, and 
$\hatt u_b(\la,x_0)\neq 0$.  If $\hatt u_b(\la,x_0)/u_b(\la,x_0)>0$, then $\hatt u_b(\lambda,\dott)$ has no 
$($resp., exactly one$)$ zero in $(x_0,b)$ if $W(u_b(\lambda,\dott), \hatt u_b(\lambda,\dott)) > 0$ $($resp., 
$W(u_b(\lambda,\dott), \hatt u_b(\lambda,\dott)) < 0$$)$.  On the other hand, 
if $\hatt u_b(\la,x_0)/u_b(\la,x_0)<0$, then $\hatt u_b$ has no $($resp., exactly one$)$ 
zero in $(x_0,b)$ if $W(u_b(\lambda,\dott), \hatt u_b(\lambda,\dott)) < 0$ 
$($resp., $W(u_b(\lambda,\dott), \hatt u_b(\lambda,\dott)) > 0$$)$.
\end{theorem}
%%%%%%%%%%%

A result analogous to Theorem \ref{t2.8}
 holds if $\tau-\lambda$ is nonoscillatory at $a$.  That is, one can establish the existence of a distinguished real-valued solution $u_a(\lambda,\dott) \neq 0$ of $(\tau - \lambda) u = 0$ which satisfies the following analog to \eqref{2.15a}:  If $\hatt u_a(\lambda,\dott)$ is any real-valued solution of $(\tau - \lambda)u = 0$ linearly independent of $u_a(\lambda,\dott)$, then
\begin{equation}
\lim_{x\downarrow a}\frac{u_a(\la,x)}{\hatt u_a(\la,x)}=0.
\end{equation}
Analogs of item $(ii)$ and $(iii)$ of Theorem \ref{t2.8}
 subsequently hold for $u_a(\lambda,\dott)$ and any real-valued solution $\hatt u_a(\lambda,\dott)$ linearly independent of $u_a(\lambda,\dott)$.

%%%%%%%%%%%
\begin{definition} \lb{d2.9} 
Assume Hypothesis \ref{h2.1} and suppose that $\lambda \in \bbR$.  If $\tau-\lambda$ is nonoscillatory at $c\in \{a,b\}$, then a nontrivial real-valued solution $u_c(\lambda,\dott)$ of $(\tau - \lambda) u = 0$ which satisfies 
\begin{equation}
\lim_{\substack{x\rightarrow c \\ x\in (a,b)}} \frac{u_c(\la,x)}{\hatt u_c(\la,x)}=0
\end{equation}
 for any other linearly independent real-valued solution $\hatt u_c(\lambda,\dott)$ of $(\tau - \lambda) u = 0$ is called a {\it principal solution} of $(\tau - \lambda) u = 0$ at $c$.  A real-valued solution of $(\tau - \lambda) u = 0$ linearly independent of a principal solution at $c$ is called a {\it nonprincipal solution} of $(\tau - \lambda) u = 0$ at $c$.
\end{definition}
%%%%%%%%%%%

If $\tau-\lambda$ is nonoscillatory at $c\in \{a,b\}$, one verifies that a principal solution at $c$ is unique up to constant multiples.  The following theorem provides a canonical method to construct principal/nonprincipal solutions.

%%%%%%%%%%%
\begin{theorem}[{\cite[Theorem 11.6]{EGNT13}}] \lb{t2.10}

Assume Hypothesis \ref{h2.1} and suppose $\tau-\lambda$ is nonoscillatory at $b$.  Let $u(\lambda,\dott) \neq 0$ be a real-valued solution of $(\tau - \lambda) u = 0$ and let $x_0\in (c,b)$ strictly exceed its last zero.  Then 
\begin{equation}\lb{11.10.13}
\hatt u_b(\la,x)=u(\la,x)\int_{x_0}^x \frac{dx'}{p(x')u(\la,x')^2}, \quad x\in (x_0,b),
\end{equation}
is a nonprincipal solution of $(\tau - \lambda) u = 0$ on $(x_0,b)$.  If, on the other hand, $u(\lambda,\dott)$ is a nonprincipal solution of $(\tau - \lambda) u = 0$, then
\begin{equation}\lb{11.10.14}
u_b(\la,x)=u(\la,x)\int_{x}^b \frac{dx'}{p(x')u(\la,x')^2}, \quad x\in (x_0,b),
\end{equation}
is a principal solution of $(\tau - \lambda) u = 0$ on $(x_0,b)$.  Analogous results hold at $a$.
\end{theorem}
%%%%%%%%%%%

Lower semiboundedness of $\dot T$ (equivalently, $T_{min}$ and its self-adjoint extensions) is a consequence of nonoscillatory behavior near the endpoints of the interval $(a,b)$, as the following theorem illustrates.

%%%%%%%%%%%
\begin{theorem}[{\cite[Theorem 11.9]{EGNT13}}] \lb{t2.11}

Assume Hypothesis \ref{h2.1}.  Suppose there exist $\lambda_a,\lambda_b\in \bbR$ such that $\tau-\lambda_a$ is nonoscillatory at $a$ and $\tau-\lambda_b$ is 
nonoscillatory at $b$.\ Then $\dot T$ and hence any self-adjoint extension of the minimal operator $T_{min}$ is lower semibounded.
\end{theorem}
%%%%%%%%%%%

When $\tau$ is regular on $(a,b)$, $\tau $ is nonoscillatory at both $a$ and $b$.  Thus, Theorem \ref{t2.11}
 immediately yields the following corollary.

%%%%%%%%%%%
\begin{corollary}[{\cite[Corollary 11.10]{EGNT13}}] \lb{c2.12}

Assume Hypothesis \ref{h2.1}.  If $\tau$ is regular 
on $(a,b)$, then $\dot T$ and hence every self-adjoint extension of $T_{min}$ 
is bounded from below.
\end{corollary}
%%%%%%%%%%%

Theorem \ref{t2.11}
 shows that nonoscillatory behavior at the endpoints implies lower semiboundedness of $\dot T$.  To state a converse result of the form ``bounded from below implies non-oscillation,'' we recall the following definition of boundedness from below at an endpoint.

%%%%%%%%%%%
\begin{definition} \lb{d2.13}
 
Assume Hypothesis \ref{h2.1}. The operator $\dot T$ is said to be 
{\it bounded from below at $a$} if there exists $c\in (a,b)$ and $\lambda_a\in \bbR$ such that 
\begin{align}
\begin{split} 
& \big(u, \dot T u\big)_{L^2((a,b);r\,dx)}\geq \lambda_a ( u, u)_{L^2((a,b);r\,dx)}, \\ 
& \quad u\in \dom\big(\dot T\big) \, \text{such that $u\equiv 0$ on $(c,b)$}.
\end{split} 
\end{align}
Similarly, $\dot T$ is said to be {\it bounded from below at $b$} if there exists $d\in (a,b)$ and $\lambda_b\in \bbR$ such that 
\begin{align}
\begin{split} 
& \big(u, \dot T u\big)_{L^2((a,b);r\,dx)}\geq \lambda_b ( u, u)_{L^2((a,b);r\,dx)}, \\ 
& \quad u\in \dom\big(\dot T\big) \, \text{such that $u\equiv 0$ on $(a,d)$}.
\end{split} 
\end{align}
\end{definition}
%%%%%%%%%%%

%%%%%%%%%%%
\begin{theorem}[{\cite[Theorem 11.13]{EGNT13}}] \lb{t2.14}

Assume Hypothesis \ref{h2.1}.  If $\dot T$ is bounded from below at $a$, then there exists $\alpha\in \bbR$ such that for all $\lambda < \alpha$, 
$\tau-\lambda$ is nonoscillatory at $a$.  An analogous result holds if $\dot T$ is bounded from below at $b$.
\end{theorem}
%%%%%%%%%%%

%%%%%%%%%%%
\begin{corollary}[{\cite[Corollary 11.14]{EGNT13}}] \lb{c2.15}

Assume Hypothesis \ref{h2.1}.  The operator $\dot T$ is bounded from below if and only if there exist $\mu\in \bbR$ and functions $g_a, g_b\in AC_{loc}((a,b))$ such that $g_a^{[1]}, g_b^{[1]}\in AC_{loc}((a,b))$, $g_a>0$ near $a$, $g_b>0$ near $b$, 
\begin{equation}\lb{11.10.55}
\begin{split}
q&\geq \mu r-s\frac{g_a^{[1]}}{g_a} + \f{\big(g_a^{[1]}\big)'}{g_a}  \, \text{ a.e.\ near $a$},\\
q&\geq \mu r-s\frac{g_b^{[1]}}{g_b} + \f{\big(g_b^{[1]}\big)'}{g_b}  \, \text{ a.e.\ near $b$}.
\end{split}
\end{equation}
\end{corollary}
%%%%%%%%%%%

%%%%%%%%%%%
\begin{definition}[{\cite[Definition 11.15]{EGNT13}}] \lb{d2.16}
 
Assume Hypothesis \ref{h2.1} and let $\lambda \in \bbR$.  The differential expression $\tau-\lambda$ is called \textit{disconjugate} on $(a,b)$ if every nontrivial real-valued solution of $\tau u = \lambda u$ has at most one zero in $(a,b)$.
\end{definition}
%%%%%%%%%%%

%%%%%%%%%%%
\begin{corollary}[{\cite[Corollary 11.16]{EGNT13}}] \lb{c2.17}
 
Assume Hypothesis \ref{h2.1}.  If $\dot T$ is bounded from below, then there exists $\alpha\in\R$ such that $(\tau -\lambda)$ is disconjugate for every $\lambda<\alpha$.  If $\tau$ is regular on $(a,b)$, then there exists $\alpha_0\in \bbR$, such that for $\lambda<\alpha_0$, each 
nontrivial solution to $(\tau-\lambda)u=0$ has at most one zero in the closed interval $[a,b]$.
\end{corollary}
%%%%%%%%%%%

Assuming $T_{min}$ is lower semibounded, the following theorem introduces generalized boundary values for functions $g\in \dom(T_{max})$.  We omit the proof since it is identical to the analogous result for three-coefficient Sturm--Liouville operators obtained in \cite[Theorem 4.5]{GLN20}.

%%%%%%%%%%%
\begin{theorem} \lb{t2.18}

Assume Hypothesis \ref{h2.1} and that $\tau$ is in the limit circle case at $a$ and $b$ $($i.e., $\tau$ is quasi-regular on $(a,b)$$)$. In addition, assume that 
$T_{min} \geq \lambda_0 I_{L^2((a,b);r dx)}$ for some $\lambda_0 \in \bbR$, and denote by 
$u_a(\lambda_0, \dott)$ and $\hatt u_a(\lambda_0, \dott)$ $($resp., $u_b(\lambda_0, \dott)$ and 
$\hatt u_b(\lambda_0, \dott)$$)$ principal and nonprincipal solutions of $(\tau -\lambda_0)u = 0$ at $a$ 
$($resp., $b$$)$, satisfying
\begin{equation}
W(\hatt u_a(\lambda_0,\dott), u_a(\lambda_0,\dott))(x) 
= W(\hatt u_b(\lambda_0,\dott), u_b(\lambda_0,\dott))(x) = 1, \quad x \in (a,b).  
\lb{11.3.33AB} 
\end{equation}
Introducing $v_j \in \dom(T_{max})$, $j\in \{1,2\}$, via 
\begin{align}
v_1(x) = \begin{cases} \hatt u_a(\lambda_0,x), & \text{for $x$ near a}, \\
\hatt u_b(\lambda_0,x), & \text{for $x$ near b},  \end{cases}   \quad 
v_2(x) = \begin{cases} u_a(\lambda_0,x), & \text{for $x$ near a}, \\
u_b(\lambda_0,x), & \text{for $x$ near b},  \end{cases}   \lb{11.3.33A}
\end{align} 
one obtains for all $g \in \dom(T_{max})$, 
\begin{align}
\begin{split} 
\wti g(a) &= - W(v_2, g)(a) = \wti g_1(a)  =  - W(u_a(\lambda_0,\dott), g)(a) 
= \lim_{x \downarrow a} \f{g(x)}{\hatt u_a(\lambda_0,x)},    \\
\wti g(b) &= - W(v_2, g)(b) = \wti g_1(b) =  - W(u_b(\lambda_0,\dott), g)(b) = \lim_{x \uparrow b} \f{g(x)}{\hatt u_b(\lambda_0,x)},    
\end{split} \lb{11.3.34A} \\
\begin{split}
{\wti g}^{\, \prime}(a) &= W(v_1, g)(a) = \wti g_2(a) = W(\hatt u_a(\lambda_0,\dott), g)(a) 
= \lim_{x \downarrow a} \f{g(x) - \wti g(a) \hatt u_a(\lambda_0,x)}{u_a(\lambda_0,x)},    \\ 
{\wti g}^{\, \prime}(b) &= W(v_1, g)(b) = \wti g_2(b) = W(\hatt u_b(\lambda_0,\dott), g)(b)
= \lim_{x \uparrow b} \f{g(x) - \wti g(b) \hatt u_b(\lambda_0,x)}{u_b(\lambda_0,x)}.   \lb{11.3.34B}
\end{split} 
\end{align}
In particular, the limits on the right-hand sides in \eqref{11.3.34A}, \eqref{11.3.34B} exist. 
\end{theorem}
%%%%%%%%%%%

%%%%%%%%%%%
\begin{remark} \lb{r2.19} 
In the case when $\tau$ is regular on $(a,b)$, the generalized boundary values defined by \eqref{11.3.34A} and \eqref{11.3.34B} coincide with the boundary values in \eqref{2.6a}:
\begin{equation}
\wti g(c) = g(c),\quad {\wti g}^{\, \prime}(c) = g^{[1]}(c),\quad c\in \{a,b\}.
\end{equation}
\hfill $\diamond$
\end{remark}
%%%%%%%%%%%

Under the assumptions of Theorem \ref{t2.18}, the self-adjoint extensions of $T_{min}$ are parametrized in terms of the generalized boundary values at the endpoints of $(a,b)$ by separated or coupled boundary conditions.  The following result is a restatement of \cite[Theorem 6.4]{EGNT13} in terms of the generalized boundary values \eqref{11.3.34A}, \eqref{11.3.34B}, so we omit its proof.

%%%%%%%%%%%
\begin{theorem} \lb{t2.20} 
Assume Hypothesis \ref{h2.1} and that $\tau$ is in the limit circle case at $a$ and $b$ $($i.e., $\tau$ is quasi-regular on $(a,b)$$)$. In addition, assume that 
$T_{min} \geq \lambda_0 I_{L^2((a,b);r dx)}$ for some $\lambda_0 \in \bbR$. Then, given 
\eqref{11.3.34A} and \eqref{11.3.34B}, the following items $(i)$ and $(ii)$ hold: \\[1mm]  
$(i)$ All self-adjoint extensions $T_{\gamma,\delta}$ of $T_{min}$ with separated boundary conditions are of the form
\begin{align}
\begin{split} 
& T_{\gamma,\delta} f = \tau f, \quad \gamma,\delta \in[0,\pi),      \lb{23.11.2.27} \\
& f \in \dom(T_{\gamma,\delta})=\left\{g\in\dom(T_{max}) \, \Bigg| \, \begin{matrix} \sin(\gamma) {\wti g}^{\, \prime}(a) 
+ \cos(\gamma) \wti g(a) = 0; \\ \sin(\delta) {\wti g}^{\, \prime}(b) + \cos(\delta) \wti g(b) = 0 \end{matrix} \right\}.    
\end{split} 
\end{align}
Moreover, $\sigma(T_{\gamma,\delta})$ is simple. \\[1mm]
$(ii)$ All self-adjoint extensions $T_{\varphi,R}$ of $T_{min}$ with coupled boundary conditions are of the type
\begin{align}
\begin{split} 
& T_{\varphi,R} f = \tau f,    \\
& f \in \dom(T_{\varphi,R})=\left\{g\in\dom(T_{max}) \, \Bigg| \begin{pmatrix} \wti g(b) 
\\ {\wti g}^{\, \prime}(b) \end{pmatrix} = e^{i\varphi}R \begin{pmatrix}
\wti g(a) \\ {\wti g}^{\, \prime}(a) \end{pmatrix} \right\}, \lb{23.11.2.27a}
\end{split}
\end{align}
where $\varphi \in [0,\pi)$, and $R \in SL(2,\bbR)$.
\end{theorem}
%%%%%%%%%%%

Under the hypotheses of Theorem \ref{t2.20}, \cite[Theorem 3.6]{EGNT13} and \eqref{11.3.34A}, \eqref{11.3.34B} imply that the minimal operator takes on the form
\begin{align}
\begin{split} 
& T_{min} f = \tau f, \\
& f \in \dom(T_{min})= \big\{g\in\dom(T_{max})  \, \big| \, \wti g(a) = {\wti g}^{\, \prime}(a) =0 
= \wti g(b) = {\wti g}^{\, \prime}(b)\big\}.      \lb{23.11.2.27b} 
\end{split}
\end{align}

%%%%%%%%%%%
\begin{remark}\lb{r2.21}

If $\tau$ is in the limit point case at one or both interval endpoints, the corresponding boundary conditions at that endpoint are dropped and only a separated boundary condition at the other end point (if the latter is a limit circle endpoint for $\tau$) has to be imposed in Theorem \ref{t2.20}
. In other words, the generalized boundary values \eqref{11.3.34A}, \eqref{11.3.34B} are only relevant if the endpoint in question is of the limit circle type. In the case where $\tau$ is in the limit point case at both endpoints, all boundary values and boundary conditions become superfluous as in this case $T_{min} = T_{max}$ is self-adjoint. 
\hfill $\diamond$
\end{remark}
%%%%%%%%%%%

If $T_{min}$ is lower semibounded, then $T_{min}$ has a \textit{Friedrichs extension}, which we shall denote by $T_F$.  The following result characterizes $T_F$ in terms of functions that mimic the behavior of principal solutions near an endpoint. 

%%%%%%%%%%%
\begin{theorem}[{\cite[Theorem 11.17]{EGNT13}}] \lb{t2.22}

Assume Hypothesis \ref{h2.1}.  If $\dot T$ is bounded from below by $\gamma_0 \in \bbR$, $\dot T \geq \gamma_0 I_{L^2((a,b);r\,dx)}$, which by Corollary \ref{c2.15}
 is equivalent to the existence of $\mu\in \bbR$ and functions $g_a$ and $g_b$ satisfying $g_a,g_b,g_a^{[1]},g_b^{[1]}\in AC_{loc}((a,b))$, $g_a>0$ a.e.\ near $a$, $g_b>0$ a.e.\ near $b$,
\begin{equation}\lb{11.10.56}
\int_a\frac{dx}{p(x)g_a(x)^2}=\int^b\frac{dx}{p(x)g_b(x)^2}=\infty,
\end{equation}
and
\begin{equation}\lb{11.10.57}
\begin{split}
q&\geq \mu r-s\frac{g_a^{[1]}}{g_a}+\frac{\big(g_a^{[1]}\big)'}{g_a} \, \text{ a.e.\ near $a$},\\
q&\geq \mu r-s\frac{g_b^{[1]}}{g_b}+\frac{\big(g_b^{[1]}\big)'}{g_b} \, \text{ a.e.\ near $b$},
\end{split}
\end{equation}
then the Friedrichs extension $T_F$ of $T_{min}$ is characterized by
\begin{align}
&T_Ff=\tau f,\no\\
&f\in \dom (T_F)=\bigg\{g\in \dom (T_{max}) \, \bigg|\, \int_a dx \, pg_a^2\bigg|\bigg(\frac{g}{g_a}\bigg)' \bigg|^2 < \infty; \lb{11.10.58}\\
&\hspace*{5.3cm} \int^b dx \, pg_b^2\bigg|\bigg(\frac{g}{g_b}\bigg)' \bigg|^2 < \infty  \bigg\}.\no
\end{align}
In particular, 
\begin{align}
&\int_a dx \, \bigg|q-\frac{\big(g_a^{[1]}\big)'}{g_a}+s\frac{g_a^{[1]}}{g_a} \bigg||f|^2 <\infty, \, 
\int^b dx \, \bigg|q-\frac{\big(g_b^{[1]}\big)'}{g_b}+s\frac{g_b^{[1]}}{g_b} \bigg| |f|^2 < \infty,    \no \\
& \hspace*{8.2cm}  f\in \dom (T_F).       \lb{11.10.59a} 
\end{align}
\end{theorem}
%%%%%%%%%%%

Theorem \ref{t2.22}
 characterizes $T_F$ in terms of functions which satisfy the same integral conditions as principal solutions (cf.~\eqref{2.16a} and \eqref{11.10.56}).  Alternatively, one may characterize $T_F$ in terms of functions for which one or both of the integrals in \eqref{11.10.56} are finite.  However, in these cases, the characterization of $T_F$ requires a certain boundary condition.

%%%%%%%%%%%
\begin{theorem}[{\cite[Theorem 11.19]{EGNT13}}]\lb{t2.23}

Assume Hypothesis \ref{h2.1}.  If $\dot T$ is bounded from below by $\gamma_0 \in \bbR$, $\dot T \geq \gamma_0 I_{L^2((a,b);r\,dx)}$, which by Corollary \ref{c2.15}
 is equivalent to the existence of $\mu\in \bbR$ and functions $h_a$ and $h_b$ satisfying $h_a,h_b,h_a^{[1]},h_b^{[1]}\in AC_{loc}((a,b))$, $h_a>0$ a.e.\ near $a$, $h_b>0$ a.e.\ near $b$,
\begin{equation}\lb{11.10.77}
\int_a\frac{dx}{p(x)h_a(x)^2}<\infty, \quad \int^b\frac{dx}{p(x)h_b(x)^2}=\infty,
\end{equation}
and
\begin{equation}\lb{11.10.78}
\begin{split}
q&\geq \mu r-s\frac{h_a^{[1]}}{h_a}+\frac{\big(h_a^{[1]}\big)'}{h_a}  \, \text{ a.e.\ near $a$},\\
q&\geq \mu r-s\frac{h_b^{[1]}}{h_b}+\frac{\big(h_b^{[1]}\big)'}{h_b}  \, \text{ a.e.\ near $b$},
\end{split}
\end{equation}
then the Friedrichs extension $T_F$ of $T_{min}$ is characterized by
\begin{align}
&T_Ff=\tau f,\no\\
& \, f\in \dom (T_F)=\bigg\{g\in \dom (T_{max}) \, \bigg|\, \int^b dx \, p h_b^2\bigg|\bigg(\frac{g}{h_b}\bigg)' \bigg|^2 < \infty;\lb{11.10.79}\\
&\hspace*{2.8cm} \int_a dx \, p h_a^2\bigg|\bigg(\frac{g}{h_a}\bigg)' \bigg|^2 < \infty; \, 
\wti g(a) = \lim_{x\downarrow a}\frac{g(x)}{h_a(x)}=0\bigg\}.\no
\end{align}
In particular,
\begin{align}
&\int_a dx \, \bigg|q-\frac{\big(h_a^{[1]}\big)'}{h_a}+s\frac{h_a^{[1]}}{h_a} \bigg||f|^2 < \infty, \, 
\int^b dx \, \bigg|q-\frac{\big(h_b^{[1]}\big)'}{h_b}+s\frac{h_b^{[1]}}{h_b} \bigg| |f|^2 < \infty,    \no \\
& \hspace*{8.2cm}  f\in \dom (T_F).     \lb{11.10.79a}
\end{align}
We omit the analogous case where the roles of $a$ and $b$ are interchanged, but note that if \eqref{11.10.77} is replaced by 
\begin{equation}\lb{11.10.79b}
\int_a\frac{dx}{p(x) h_a(x)^2}<\infty, \quad \int^b\frac{dx}{p(x) h_b(x)^2}<\infty,
\end{equation}
one obtains 
\begin{align}
&T_Ff=\tau f,  \no \\
& \, f\in \dom (T_F)=\bigg\{g\in \dom (T_{max}) \, \bigg|\, \int_a dx \, p h_a^2\bigg|\bigg(\frac{g}{h_a}\bigg)' \bigg|^2 < \infty;    \lb{11.10.79c} \\ 
& \quad  \int^b dx \, p h_b^2\bigg|\bigg(\frac{g}{h_b}\bigg)' \bigg|^2 < \infty; \, \wti g(a) = \lim_{x\downarrow a}\frac{g(x)}{h_a(x)}=0, \, 
\wti g(b) = \lim_{x\uparrow b}\frac{g(x)}{h_b(x)}=0 \bigg\}.\no
\end{align}
Here $\wti g(a)$ and $\wti g(b)$ are consistent with their use in \eqref{11.3.34A}.  
\end{theorem}
%%%%%%%%%%%

As a self-adjoint extension of $T_{min}$, the Friedrichs extension $T_F$ may be characterized according to Theorem \ref{t2.20}.  In fact, \eqref{11.10.79c} implies $T_F \subseteq T_{0,0}$ and since self-adjoint operators are maximal, the Friedrichs extension of $T_{min}$ corresponds to separated boundary conditions with $\gamma=\delta=0$, that is, 
\begin{equation} 
T_F = T_{0,0}.
\end{equation}
 
%%%%%%%%%%%
\begin{corollary} \lb{c2.24}
Assume Hypothesis \ref{h2.1} and that $\tau$ is in the limit circle case at $a$ and $b$ $($i.e., $\tau$ is quasi-regular on $(a,b)$$)$.  If $T_{min}$ is lower semibounded, then the Friedrichs extension $T_F$ of $T_{min}$ equals $T_{0,0}$ $($cf.\ Theorem \ref{t2.20}
\,$(i)$$)$ and hence is of the form  
\begin{align}
\begin{split} 
&T_Ff=\tau f,   \\
&\, f\in \dom (T_F)= \big\{g\in \dom (T_{max}) \, \big|\, \wti g(a) = 0 = \wti g(b)\big\}. \lb{11.10.100A}
\end{split} 
\end{align}
Moreover, if $h_a, h_b$ are as in Theorem \ref{t2.23}
 $($e.g., if $h_a = \hatt u_a$, $h_b = \hatt u_b$ are nonprincipal solutions of $(\tau - \lambda)u = 0$, $\lambda \leq \gamma_0$$)$ then 
\begin{equation}
\int^b dx \, p(x) h_b(x)^2\bigg|\bigg(\frac{g(x)}{h_b(x)}\bigg)' \bigg|^2 < \infty \, \text{ and } \, 
\int_a dx \, p(x) h_a(x)^2\bigg|\bigg(\frac{g(x)}{h_a(x)}\bigg)' \bigg|^2 < \infty.  \lb{11.10.100B}
\end{equation} 
In particular, if $\tau$ is regular on $(a,b)$, then the Friedrichs extension $T_F$ of $T_{min}$ is of the form  
\begin{align}
\begin{split} 
&T_Ff=\tau f,   \\
&\, f\in \dom (T_F) = \{g\in \dom (T_{max}) \, |\, g(a) = 0 = g(b)\}. \lb{11.10.101}
\end{split} 
\end{align}
\end{corollary}
%%%%%%%%%%%
\begin{proof}
It suffices to show \eqref{11.10.100B} for all $g\in \dom(T_{max})$.  Indeed, the characterization of $T_F$ in \eqref{11.10.100A} then follows at once from \eqref{11.10.79c}.  We will prove convergence of the first integral in \eqref{11.10.100B}; the second integral is treated in an entirely analogous manner.  To this end, let $g\in \dom(T_{max})$.  Since $\tau$ is in the limit circle case at both endpoints of $(a,b)$, it follows that $h_a,h_b\in \dom(T_{max})$.  In consequence, by \eqref{2.10a}, $W(h_c,g) = g^{[1]}h_c - g h_c^{[1]}$, $c\in\{a,b\}$, is bounded on $(a,b)$.  Therefore, by \eqref{11.10.79b},
\begin{align}
\int_a dx \, p h_a^2\bigg|\bigg(\frac{g}{h_a}\bigg)' \bigg|^2 
&= \int_a dx \, p h_a^2\bigg|\frac{h_a g' - g h_a'}{h_a^2} \bigg|^2 \no  \\
&= \int_a dx \, \frac{1}{ph_a^2} \big|g^{[1]} h_a - g h_a^{[1]}\big|^2 
<\infty.
\end{align}
Finally, \eqref{11.10.101} follows immediately from \eqref{11.10.100A} and Remark \ref{r2.19}
.
\end{proof}
%%%%%%%%%%%

The case in Corollary \ref{c2.24} where $\tau$ is regular on $(a,b)$ was treated in \cite[Corollary~11.20]{EGNT13}.

%%%%%%%%%%%
\begin{remark} \lb{r2.25} 
If $T_{min}$ is lower semibounded and $\tau$ is in the limit circle case at $a$ and in the limit point case at $b$, then $T_F$ requires no boundary condition at the limit point endpoint $b$ (cf.~Remark \ref{r2.21}
).  In this case, $T_F$ takes the form
\begin{align}
T_Ff=\tau f,\quad f\in \dom (T_F)= \big\{g\in \dom (T_{max}) \, \big|\, \wti g(a) = 0\big\}. \lb{11.10.100Aa}
\end{align}
An analogous characterization holds in the case when $\tau$ is in the limit point case at $a$ and in the limit circle case at $b$.  If $\tau$ is in the limit point case at both endpoints of $(a,b)$, then $T_{min}$ is self-adjoint and $T_F=T_{min}$.\hfill $\diamond$
\end{remark}
%%%%%%%%%%%

%%%%%%%%%%%%%%%%%%%%%%%%%%%%%%%
%%%%%%%%%%%%%%%%%%%%%%%%%%%%%%%
\section{Strict Domain Monotonicity of the Principal Eigenvalue \\ of the Friedrichs Extension} \lb{s3}
%%%%%%%%%%%%%%%%%%%%%%%%%%%%%%%
%%%%%%%%%%%%%%%%%%%%%%%%%%%%%%%

We shall assume throughout this section that $\tau$ is \textit{regular} on $(a,b)$.  In this case, $T_{min}$ is lower semibounded by Corollary \ref{c2.12}, and its Friedrichs extension $T_F$ is characterized by Dirichlet boundary conditions according to \eqref{11.10.101}.  The sesquilinear form $Q_F$ uniquely associated to $T_F$ is given by (cf.~\cite[Equation (13.6)]{EGNT13})
\begin{align}
Q_F(f,g) &= \int_a^b dx\, \Big[p(x)^{-1}\overline{f^{[1]}(x)}g^{[1]}(x) + q(x) \overline{f(x)} g(x)\Big]\lb{2.14}\\
&= \big(|T_F|^{1/2}f,\sgn(T_F)|T_F|^{1/2}g\big)_{L^2((a,b); r\,dx)},\no
\end{align}
\begin{align}
f, g \in \dom(Q_F)&= \Bigg\{h\in L^2((a,b); r\,dx) \, \Bigg| \, \begin{aligned}&h \in AC ([a,b]), \, h(a) = h(b) = 0,\\
&(rp)^{-1/2} h^{[1]} \in L^2((a,b); r\,dx)\end{aligned}\Bigg\}\lb{2.15}\\
&=\dom\big(|T_F|^{1/2}\big).\no
\end{align}

In the regular case, every self-adjoint extension of $T_{min}$ has a discrete spectrum consisting of isolated eigenvalues with finite (geometric) multiplicities, see \cite[Corollary 7.2]{EGNT13}.  Actually, in the case of separated boundary conditions, all eigenvalues are simple by Theorem \ref{t2.20}
 $(i)$ (cf.~\cite[Corollary 7.4]{EGNT13}); that is, all eigenvalues have multiplicity one.  In particular, each eigenvalue of the Friedrichs extension $T_F$ of $T_{min}$ is simple.  Next, we turn to strict domain monotonicity of the principal (i.e., smallest) eigenvalue of the Friedrichs extension, $T_{F,(a_1,b_1)}$, $-\infty<a_1<b_1<\infty$, where we temporarily indicate the dependence of all relevant objects ($T_F$, $Q_F(\,\cdot\,,\,\cdot\,)$, etc.) on the underlying finite interval $(a_1,b_1)$ by appending the interval as a subscript ($T_{F,(a_1,b_1)}$, $Q_{F,(a_1,b_1)}(\,\cdot\,,\,\cdot\,)$, etc.).  We shall denote the corresponding principal eigenvalue of $T_{F,(a_1,b_1)}$ by 
 $\lambda_1(T_{F,(a_1,b_1)})$.

The proof of strict domain monotonicity of the principal eigenvalue $\lambda_1(T_{F,(a,b)})$ of the Friedrichs extension $T_{F, (a,b)}$ of $T_{min, (a,b)}$ in Theorem \ref{t3.4} below, relies on the variational characterization of the principal eigenvalue in terms of Rayleigh quotients.  As such, it is convenient to introduce the following notation for the Rayleigh quotient of a lower semibounded self-adjoint operator $T$ in a Hilbert space $\cH$:
\begin{equation}
R_T(f;\cH) := \frac{\big(|T|^{1/2}f,\sgn(T)|T|^{1/2}f\big)_{\cH}}{\|f\|_{\cH}^2},\quad f\in \dom\big(|T|^{1/2}\big)\big\backslash\{0\}.
\end{equation}

We start by recalling Rayleigh's principle:

%%%%%%%
\begin{theorem} \lb{t3.1} 
Suppose $A$ is self-adjoint and bounded from below in $\cH$, that is, for some $c \in \bbR$, $A \geq c I_{\cH}$. In addition, assume that $\lambda_{0,ess} = \inf (\sigma_{ess}(A)) > c$ and that 
$\sigma_d(A) \cap [c, \lambda_{0,ess}) \neq \emptyset$. Denoting the eigenvalues of $A$ below $\lambda_{0,ess}$ 
by $\{\lambda_j(A)\}_{j \in J}$, and the associated orthogonal eigenfunctions by $u_j$, $J \subseteq \bbN$ an appropriate $($finite or infinite\,$)$ index set, using the natural ordering $\lambda_1(A) \leq \lambda_2(A) \leq \cdots$ $($taking multiplicities into account\,$)$, one obtains
\begin{align}
\lambda_1(A) = \underset{0 \neq f \in \dom(A)}{\min} \f{(f,A f)_{\cH}}{\|f\|^2_{\cH}} =  \underset{0 \neq f \in \dom(|A|^{1/2})}{\min} R_A(f;\cH),   \lb{C.4.1}
\end{align}
and, if $|J| \geq 2$,
\begin{align}
\begin{split}
\lambda_j(A) = \underset{\substack{0 \neq f \in \dom(A) \\ (f, u_k)_{\cH} = 0, \, 1 \leq k \leq j-1}}{\min} \f{(f,A f)_{\cH}}{\|f\|^2_{\cH}} =  \underset{\substack{0 \neq f \in \dom(|A|^{1/2}) \\ (f, u_k)_{\cH} = 0, \, 1 \leq k \leq j-1}}{\min} 
R_A(f;\cH),& \\
j \in J, \; j \geq 2.&    \lb{C.4.2}
\end{split}
\end{align}
\end{theorem}
%%%%%%%

Here, it is of course understood that $\lambda_{0,ess} = +\infty$ if $\sigma_{ess}(A) = \emptyset$.  The first parts in \eqref{C.4.1} and \eqref{C.4.2} are proven via the spectral theorem for $A$ in \cite[Theorem~1.1]{WS72}; their quadratic form versions---the second parts in \eqref{C.4.1} and \eqref{C.4.2}---are proven in precisely the same manner. 

The following related result will be employed in connection with strict domain monotonicity of the lowest eigenvalue of regular Sturm--Liouville operators with Dirichlet boundary conditions in Theorem \ref{t3.4}.

%%%%%%%
\begin{theorem} \lb{t3.2} 
Suppose $A$ is self-adjoint and bounded from below in $\cH$, that is, for some $c \in \bbR$, $A \geq c I_{\cH}$. In addition, assume that $\lambda_{0,ess}(A) = \inf (\sigma_{ess}(A)) \in( c, \infty) \cup \{\infty\}$ and that 
$\sigma_d(A) \cap [c, \lambda_{0,ess}(A)) \neq \emptyset$. We denote the eigenvalues of $A$ below 
$\lambda_{0,ess}(A)$ by $\{\lambda_j(A)\}_{j \in J}$, and the associated orthogonal eigenfunctions by $u_j$, 
$J \subseteq \bbN$ an appropriate $($finite or infinite\,$)$ index set, using the natural ordering $\lambda_1(A) \leq \lambda_2(A) \leq \cdots$ $($taking multiplicities into account\,$)$. If there exist vectors $v_j \in \dom(A)\backslash\{0\}$, respectively, 
$v_j \in \dom\big(|A|^{1/2}\big) \big \backslash\{0\}$, $j \in J$, such that
\begin{equation}
\mu_1(A) = \f{(v_1,A v_1)_{\cH}}{\|v_1\|^2_{\cH}} 
= \underset{0 \neq f \in \dom(A)}{\min} \f{(f,A f)_{\cH}}{\|f\|^2_{\cH}},  \lb{C.4.3}
\end{equation}
and, if $|J| \geq 2$,
\begin{equation}
\mu_j(A) = \f{(v_j,A v_j)_{\cH}}{\|v_j\|^2_{\cH}}
= \underset{\substack{0 \neq f \in \dom(A) \\ (f, v_k)_{\cH} = 0, \, 1 \leq k \leq j-1}}{\min} \f{(f,A f)_{\cH}}{\|f\|^2_{\cH}}, 
\quad j \in J, \; j \geq 2,     \lb{C.4.4} 
\end{equation}
respectively, 
\begin{align}
\mu_1(A) = R_A(v_1;\cH) =  \underset{0 \neq f \in \dom(|A|^{1/2})}{\min} R_A(f;\cH),    \lb{C.4.5} 
\end{align}
and, if $|J| \geq 2$,
\begin{align}
\mu_j(A) = R_A(v_j;\cH) =  \underset{\substack{0 \neq f \in \dom(|A|^{1/2}) \\ (f, v_k)_{\cH} = 0, \, 1 \leq k \leq j-1}}{\min} 
R_A(f;\cH),  \quad j \in J, \; j \geq 2,     \lb{C.4.6}
\end{align}
then $\mu_j(A) = \lambda_j(A)$ and $v_j$ is a corresponding eigenvector, $j \in J$. In particular, in the quadratic form context, where $v_j \in \dom\big(|A|^{1/2}\big)$, $v_j$ also satisfies $v_j \in \dom(A)$, $j \in J$. 
\end{theorem}
%%%%%%%
\begin{proof}
Since \eqref{C.4.3} and \eqref{C.4.4} are proven in \cite[Theorem~1.2]{WS72}, we focus on the quadratic form versions \eqref{C.4.5}, \eqref{C.4.6}.

Since $\mu_1(A)$ is a minimum, one obtains 
\begin{align}
\begin{split} 
\big(|A|^{1/2} (v_1 + \varepsilon w), \sgn(A) (|A|^{1/2} (v_1 + \varepsilon w)\big)_{\cH} \geq 
\mu_1 ((v_1 + \varepsilon w), (v_1 + \varepsilon w))_{\cH},& \\ 
\varepsilon \in \bbR, \; w \in \dom\big(|A|^{1/2}\big),&
\end{split}
\end{align}
and hence, using $\mu_1(A) = \big(|A|^{1/2}v_1, \sgn(A) |A|^{1/2} v_1\big)_{\cH}\big/\|v_1\|^2_{\cH}$, one obtains
\begin{align}
0 & \leq \varepsilon^2 \big[\big(|A|^{1/2} w, \sgn(A) |A|^{1/2} w\big)_{\cH} - \mu_1(A) \|w\|_{\cH}^2\big]   \no \\
& \quad + 2 \varepsilon \Re\big(\big(|A|^{1/2} v_1, \sgn(A) |A|^{1/2} w\big)_{\cH}  - \mu_1(A) (v_1,w)_{\cH}\big).  \no 
\end{align}
Since $\varepsilon \in \bbR$ was arbitrary, one concludes with $c' < c$, that
\begin{align} 
0 & = \big(|A|^{1/2} v_1, \sgn(A) (|A|^{1/2} w\big)_{\cH} - \mu_1(A) (v_1,w)_{\cH}    \no \\
& = \big((A - c' I_{\cH})^{1/2} v_1, (A - c' I_{\cH})^{1/2} w\big)_{\cH}    \no \\
& \quad + [c' - \mu_1(A)] \big((A - c' I_{\cH})^{-1/2} v_1, (A - c' I_{\cH})^{1/2} w\big)_{\cH}
\end{align}
and since $\ran\big( (A - c' I_{\cH})^{1/2}\big) = \cH$ (due to the choice $c' < c$), one obtains
\begin{equation}
(A - c' I_{\cH})^{1/2} v_1 +  [c' - \mu_1(A)] (A - c' I_{\cH})^{-1/2} v_1 = 0.
\end{equation}
Thus, $(A - c' I_{\cH})^{1/2} v_1 \in \dom\big(|A|^{1/2}\big)$, equivalently, $v_1 \in \dom(A)$, and hence 
\begin{equation}
A v_1 = \mu_1(A) v_1.
\end{equation}
Since $\mu_1(A)$ is a discrete eigenvalue of $A$, it follows that $\lambda_1(A) \leq \mu_1(A)$. On the other hand, 
\begin{align}
\mu_1(A) = \underset{0 \neq f \in \dom(|A|^{1/2})}{\min} R_A(f;\cH) \leq R_A(u_1;\cH) = \lambda_1(A),
\end{align}
and hence $\mu_1(A) = \lambda_1(A)$.

Next, suppose $j \in J$, $j \geq 2$, is fixed, and $\mu_k(A)$ are eigenvalues of $A$ with corresponding eigenvectors $v_k$, $1 \leq k \leq j-1$, that is, $A v_k = \mu_k(A) v_k$, $1 \leq k \leq j-1$. By construction, 
\begin{equation}
(v_k, v_{\ell})_{\cH} = \|v_k\|_{\cH}^2 \, \delta_{k,\ell}, \quad 1 \leq k, \ell \leq j.
\end{equation}

Assuming $\varepsilon \in \bbR$ and 
\begin{equation}
w \in \dom\big(|A|^{1/2}\big), \quad (w,v_k)_{\cH} =0, \; 1 \leq k \leq j-1,    \lb{C.4.11} 
\end{equation}
one again concludes
\begin{equation} 
\big(|A|^{1/2} (v_j + \varepsilon w), \sgn(A) (|A|^{1/2} (v_j + \varepsilon w)\big)_{\cH} \geq 
\mu_j ((v_j + \varepsilon w), (v_1 + \varepsilon w))_{\cH}
\end{equation}
and hence once more,
\begin{align} 
0 & = \big(|A|^{1/2} v_j, \sgn(A) (|A|^{1/2} w\big)_{\cH} - \mu_j(A) (v_1,w)_{\cH}    \no \\
& = \big((A - c' I_{\cH})^{1/2} v_j, (A - c' I_{\cH})^{1/2} w\big)_{\cH}    \no \\
& \quad + [c' - \mu_j(A)] \big((A - c' I_{\cH})^{-1/2} v_j, (A - c' I_{\cH})^{1/2} w\big)_{\cH}
\end{align}
for all $w \in \cH$ satisfying the conditions in \eqref{C.4.11}. Since
\begin{equation}
0 = (w,v_k)_{\cH} = \big((A - c' I_{\cH})^{1/2} w, (A - c' I_{\cH})^{-1/2} v_k\big)_{\cH}, 
\quad 1 \leq k \leq j-1,
\end{equation}
one concludes that
\begin{equation}
(A - c' I_{\cH})^{1/2} v_j + [c' - \mu_j(A)] (A - c' I_{\cH})^{-1/2} v_j 
= \sum_{k=1}^{j-1} c_k (A - c' I_{\cH})^{-1/2} v_k   \lb{C.4.16}
\end{equation}
for some $c_k \in \bbC$. Taking the scalar product of both sides in \eqref{C.4.16} with the 
element $(A - c' I_{\cH})^{1/2} v_\ell$ yields
\begin{align}
c_{\ell} &= \big((A - c' I_{\cH})^{1/2} v_{\ell}, (A - c' I_{\cH})^{1/2} v_j \big)_{\cH}     \no \\
& \quad + [c' - \mu_j(A)] \big((A - c' I_{\cH})^{1/2} v_{\ell}, (A - c' I_{\cH})^{-1/2} v_j\big)_{\cH}    \no \\
&= \big((A - c' I_{\cH})^{1/2} v_{\ell}, (A - c' I_{\cH})^{1/2} v_j \big)_{\cH} + [c' - \mu_j(A)] (v_{\ell}, v_j\big)_{\cH}   \no \\
&= \big((A - c' I_{\cH})^{1/2} v_{\ell}, (A - c' I_{\cH})^{1/2} v_j \big)_{\cH} \no \\
&= \big(|A|^{1/2} v_{\ell}, \sgn(A) |A|^{1/2} v_j\big)_{\cH}  - c' (v_{\ell}, v_j)_{\cH}    \no \\
&= \big(|A|^{1/2} v_{\ell}, \sgn(A) |A|^{1/2} v_j\big)_{\cH} \no \\ 
&= \mu_{\ell}(A) (v_{\ell}, v_j)_{\cH} = 0, \quad 1 \leq \ell \leq j-1,
\end{align}
and hence results in
\begin{equation}
(A - c' I_{\cH})^{1/2} v_j +  [c' - \mu_j(A)] (A - c' I_{\cH})^{-1/2} v_j = 0.
\end{equation}
Once more this implies 
$(A - c' I_{\cH})^{1/2} v_j \in \dom\big(|A|^{1/2}\big)$, equivalently, $v_j \in \dom(A)$, and hence 
\begin{equation}
A v_j = \mu_j(A) v_j.
\end{equation}
Finally, suppose that $\mu_k(A) = \lambda_k(A)$, $1 \leq k \leq j-1$. Since $\mu_j(A)$ is a discrete eigenvalue of $A$, it follows that $\lambda_j(A) \leq \mu_j(A)$. On the other hand, 
\begin{align}
\mu_j(A) &= \underset{\substack{0 \neq f \in \dom(|A|^{1/2}) \\ (f, v_k)_{\cH} = 0, \, 1 \leq k \leq j-1}}{\min} 
R_A(f;\cH)     \no \\
&= \underset{\substack{0 \neq f \in \dom(|A|^{1/2}) \\ (f, u_k)_{\cH} = 0, \, 1 \leq k \leq j-1}}{\min} 
R_A(f;\cH)     \no \\
&\leq R_A(u_j;\cH) = \lambda_j(A),
\end{align}
and hence $\mu_j(A) = \lambda_j(A)$. 
\end{proof}
%%%%%%%

%%%%%%%
\begin{remark} \lb{r3.3} 
Compared to the min--max (resp., max--min) theorem, Theorems \ref{t3.1}
 and \ref{t3.2}
 have the drawback that they directly involve {\it a priori} knowledge of the eigenvectors $u_j$, or, of the vectors $v_j$, $j \in J$, rather than $j$-dimensional (resp., $j-1$-dimensional) subspaces, $j \geq 2$. Still, Theorem \ref{t3.2}
 is instrumental in proving strict domain monotonicity for 
$\lambda_1(A)$ (as opposed to just domain monotonicity) as will be illustrated in the proof of Theorem \ref{t3.4} below. \hfill $\diamond$
\end{remark}
%%%%%%%

%%%%%%%%%%%
\begin{theorem} \lb{t3.4} 
Assume Hypothesis \ref{h2.1} and suppose that $\tau$ is regular on $(a,b)$.  Let $T_{F,(c,d)}$ denote the Friedrichs extension of $T_{min,(c,d)}$, where $(c,d)\subset (a,b)$ and $c<d$; that is, either $a\leq c$ and $d<b$, or else $a<c$ and $d\leq b$, so that $(c,d)$ is strictly contained in $(a,b)$.  Then
\begin{equation}
\lambda_1(T_{F,(a,b)}) < \lambda_1(T_{F,(c,d)}).
\end{equation}
\end{theorem}
%%%%%%%%%%%
\begin{proof}
For simplicity we shall assume that $a\leq c$ and $d<b$.  The case $a<c$ and $d\leq b$ is analogous.  We begin with the following simple observation.  If $f\in \dom\big(|T_{F,(c,d)}|^{1/2}\big)$, then the characterization in \eqref{2.15} implies the function $\widetilde{f}$ defined on $(a,b)$ by
\begin{equation}\lb{2.17}
\widetilde{f}(x)=
\begin{cases}
f(x),&x\in(c,d),\\
0,&x\in (a,b)\backslash (c,d),
\end{cases}
\end{equation}
belongs to $\dom\big(|T_{F,(a,b)}|^{1/2}\big)$ and \eqref{2.14} yields
\begin{equation}\lb{2.18}
Q_{F,(c,d)}(f,f) = Q_{F,(a,b)}\big(\widetilde{f},\widetilde{f}\big).
\end{equation}
Hence, by the variational characterization of the principal eigenvalue (cf.~Theorem \ref{t3.1}),
\begin{align}
\lambda_1(T_{F,(a,b)})&= \min_{0\neq f\in \dom(|T_{F,(a,b)}|^{1/2})}R_{T_{F,(a,b)}}\big(f;L^2((a,b);r\,dx)\big)\no\\
&\leq \min_{0\neq f\in \dom(|T_{F,(c,d)}|^{1/2})}R_{T_{F,(c,d)}}\big(f;L^2((c,d);r\,dx)\big)\no\\
&= \lambda_1(T_{F,(c,d)}).\lb{2.19}
\end{align}

To prove the inequality in \eqref{2.19} is strict, we argue by way of contradiction and assume that $\lambda_1(T_{F,(a,b)}) = \lambda_1(T_{F,(c,d)})$.  Let $u_{1,F,(c,d)}\in \dom(T_{F,(c,d)})$ denote the eigenfunction (unique up to nonzero constant multiples) corresponding to the principal eigenvalue $\lambda_1(T_{F,(c,d)})$ of $T_{F,(c,d)}$ and introduce the test function $v_{1,F,(a,b)}$ by
\begin{equation}
v_{1,F,(a,b)}(x) := \widetilde{u}_{1,F,(c,d)}(x) = 
\begin{cases}
u_{1,F,(c,d)}(x),& x\in(c,d),\\
0,& x\in(a,b)\backslash(c,d).
\end{cases}
\end{equation}
By the observation at the beginning of the proof,
\begin{equation}
v_{1,F,(a,b)}\in \dom(Q_{F,(a,b)}) = \dom\big(|T_{F,(a,b)}|^{1/2} \big),
\end{equation}
and \eqref{2.18} yields
\begin{align}
\lambda_1(T_{F,(a,b)}) &= R_{T_{F,(c,d)}}\big(u_{1,F,(c,d)};L^2((c,d);r\,dx)\big) \\  
&= R_{T_{F,(a,b)}}\big(v_{1,F,(a,b)};L^2((a,b);r\,dx)\big)\no\\
&= \lambda_1(T_{F,(c,d)}) \no\\
&= \min_{0\neq f\in \dom(|T_{F,(a,b)}|^{1/2})}R_{T_{F,(a,b)}}\big(f;L^2((a,b);r\,dx)\big).\no
\end{align}
Therefore, $v_{1,F,(a,b)}$ is a minimizer of the Rayleigh quotient for $T_{F,(a,b)}$.  In consequence, by Theorem \ref{t3.2}, $v_{1,F,(a,b)}$ is an eigenfunction of $T_{F,(a,b)}$ corresponding to the eigenvalue 
$\lambda_1(T_{F,(a,b)})$.  In particular, $(\tau - \lambda_1(T_{F,(a,b)}))v_{1,F,(a,b)}=0$ on $(a,b)$.  However, since $v_{1,F,(a,b)}$ vanishes identically on the open interval $(d,b)$, Theorem \ref{t2.3}
 implies $v_{1,F,(a,b)} \equiv 0$, which is a contradiction.
\end{proof}
%%%%%%%%%%%

%%%%%%%
\begin{remark} \lb{r3.5}
One recognizes the special case of a unique continuation argument used in the final part of the proof of Theorem \ref{t3.4}. For an extension of Theorem \ref{t3.4}
 to multi-dimensional Schr\"odinger operators on open domains $\Omega \subset \bbR^n$ with Dirichlet boundary conditions on $\partial \Omega$ under a natural spectral assumption and with Lebesgue measure replaced by Bessel capacity, see \cite{GZ94}; we will return to the issue of Bessel-type capacities in the present four-coefficient Sturm--Liouville operator context elsewhere. \hfill $\diamond$
\end{remark}
%%%%%%%

%%%%%%%%%%%%%%%%%%%%%%%%%%%%%%%
%%%%%%%%%%%%%%%%%%%%%%%%%%%%%%%
\section{Characterization of the Lower Bounds of the Friedrichs Extension in Terms of Strictly Positive Solutions} \lb{s4}
%%%%%%%%%%%%%%%%%%%%%%%%%%%%%%%
%%%%%%%%%%%%%%%%%%%%%%%%%%%%%%%

Returning to the general singular context, in this section we provide the characterization of lower bounds for $T_F$,  that is, 
$T_F\geq \lambda_0I_{L^2((a,b);r\,dx)}$ (equivalently, $T_{min}\geq \lambda_0I_{L^2((a,b);r\,dx)}$) in terms of the existence of a strictly positive solution $u_0(\lambda_0,\,\cdot\,)>0$ of $(\tau - \lambda_0)u = 0$ on $(a,b)$.

We begin by recalling an elementary formula for the construction of a second linearly independent solution of $(\tau - \lambda)u=0$.

%%%%%%%%%%%
\begin{lemma} \lb{l4.1} 
Assume Hypothesis \ref{h2.1}.  Suppose $\lambda \in \bbR$, $(c,d)\subseteq (a,b)$, and $x_0\in (c,d)$.  If $u_1(\lambda,\,\cdot\,)$ is a solution of $(\tau - \lambda)u = 0$ on $(c,d)$ that does not vanish in $(c,d)$, then $u_2(\lambda,\,\cdot\,)$ defined by
\begin{equation}
u_2(\lambda,x) = u_1(\lambda,x)\int_{x_0}^x\frac{dt}{p(t)u_1(\lambda,t)^2},\quad x\in (c,d),
\end{equation}
is also a solution of $(\tau - \lambda) u = 0$ on $(c,d)$.  Moreover, $u_1(\lambda,\,\cdot\,)$ and $u_2(\lambda,\,\cdot\,)$ are linearly independent on $(c,d)$.
\end{lemma}
%%%%%%%%%%%

%%%%%%%%%%%
\begin{theorem} \lb{t4.2} 
Assume Hypothesis \ref{h2.1} and let $\lambda\in \bbR$.  The following statements $(i)$--$(iv)$ are equivalent.\\[1mm]
$(i)$  $\tau - \lambda$ is disconjugate on $(a,b)$.\\[1mm]
$(ii)$ For every pair of points $x_1,x_2\in (a,b)$, $x_1\neq x_2$, and arbitrary $u_1,u_2\in \bbR$, there exists a unique real-valued solution $u_*(\lambda,\,\cdot\,)$ of $(\tau - \lambda)u = 0$ on $(a,b)$ such that
\begin{equation}\lb{3.1}
u_*(\lambda,x_1)=u_1\quad \text{and}\quad u_*(\lambda,x_2)=u_2.
\end{equation}
$(iii)$  Every pair of real-valued linearly independent solutions $u_j(\lambda,\,\cdot\,)$, $j\in \{1,2\}$, of $(\tau - \lambda)u = 0$ on $(a,b)$ satisfies
\begin{equation}\lb{3.2}
u_1(\lambda,x_1)u_2(\lambda,x_2) - u_1(\lambda,x_2)u_2(\lambda,x_1) \neq 0
\end{equation}
for every pair of points $x_1,x_2\in (a,b)$ with $x_1\neq x_2$.\\[1mm]
$(iv)$ There exists a strictly positive solution $u_0(\lambda,\,\cdot\,)>0$ of $(\tau - \lambda)u = 0$ on $(a,b)$.
\end{theorem}
%%%%%%%%%%%
\begin{proof}
To prove the equivalence of statements $(i)$--$(iii)$, let $u_j(\lambda,\,\cdot\,)$, $j\in\{1,2\}$, denote linearly independent real-valued solutions of $(\tau - \lambda)u = 0$ on $(a,b)$.  Any real-valued solution $u_*(\lambda,\,\cdot\,)$ of $(\tau - \lambda)u = 0$ is then of the form
\begin{equation}\lb{3.3}
u_*(\lambda,\,\cdot\,) = c_1u_1(\lambda,\,\cdot\,) + c_2u_2(\lambda,\,\cdot\,),\quad c_j\in \bbR,\, j\in \{1,2\}.
\end{equation}
Let $x_1,x_2\in (a,b)$, $x_1\neq x_2$.  The system
\begin{equation}\lb{3.4}
\begin{pmatrix}
u_1(\lambda,x_1)	&		u_2(\lambda,x_1)	\\
u_1(\lambda,x_2)	&		u_2(\lambda,x_2)	\\
\end{pmatrix}
\begin{pmatrix}
c_1\\
c_2
\end{pmatrix}
=
\begin{pmatrix}
u_1\\
u_2
\end{pmatrix}
\end{equation}
has a unique solution for all values $u_j\in \bbR$, $j\in \{1,2\}$, if and only if \eqref{3.2} holds.  This yields the equivalence of statements $(ii)$ and $(iii)$.  Alternatively, \eqref{3.4} has a unique solution for all values $u_j\in \bbR$, $j\in \{1,2\}$, if and only if
\begin{equation}\lb{3.5}
\begin{pmatrix}
u_1(\lambda,x_1)	&		u_2(\lambda,x_1)	\\
u_1(\lambda,x_2)	&		u_2(\lambda,x_2)	\\
\end{pmatrix}
\begin{pmatrix}
c_1\\
c_2
\end{pmatrix}
=
\begin{pmatrix}
0\\
0
\end{pmatrix}
\,\text{ implies $c_1=c_2=0$.}
\end{equation}
In turn, \eqref{3.5} holds if and only if the only solution $u_*(\lambda,\,\cdot\,)$ of $(\tau - \lambda)u = 0$ on $(a,b)$ that has zeros at $x_1$ and $x_2$ is $u_*(\lambda,\,\cdot\,)\equiv 0$.  This yields the equivalence of statements $(i)$ and $(ii)$.

The proof can be completed by establishing the equivalence of statements $(i)$ and $(iv)$.  To prove that $(i)$ implies $(iv)$, we argue by way of contradiction.  Assume $\tau - \lambda$ is disconjugate on $(a,b)$ and that every real-valued solution of $(\tau - \lambda)u = 0$ on $(a,b)$ has at least one zero in $(a,b)$.  Since $\tau - \lambda$ is disconjugate, every nontrivial real-valued solution of $(\tau - \lambda)u = 0$ has at most one zero in $(a,b)$.  Thus, every nontrivial real-valued solution of $(\tau - \lambda)u = 0$ on $(a,b)$ has exactly one zero in $(a,b)$.  Choose linearly independent real-valued solutions $u_j(\lambda,\,\cdot\,)$, $j\in \{1,2\}$, of $(\tau - \lambda)u = 0$ on $(a,b)$ and distinct points $x_1,x_2\in (a,b)$ with
\begin{equation}\lb{3.7}
\begin{split}
&u_1(\lambda,x_1)=0,\quad u_1^{[1]}(\lambda,x_1)=u_1(\lambda,x_2),\\
&u_2(\lambda,x_2)=0,\quad u_2^{[1]}(\lambda,x_2)=u_2(\lambda,x_1).
\end{split}
\end{equation}
Note that $u_1(\lambda,x_2)\neq 0$ and $u_2(\lambda,x_1)\neq 0$ by hypothesis.  Since the $u_j(\lambda,\,\cdot\,)$, $j\in \{1,2\}$, are linearly independent solutions, their Wronskian (cf.~\eqref{2.4}) is a nonzero constant, say $W(u_1(\lambda,\,\cdot\,),u_2(\lambda,\,\cdot\,))(x)=c\neq 0$, $x\in (a,b)$.  Separately evaluating the Wronskian at $x=x_1$ and $x=x_2$, one obtains, by \eqref{3.7},
\begin{equation}\lb{3.8}
\begin{split}
&0 \neq c = W(u_1(\lambda,\,\cdot\,),u_2(\lambda,\,\cdot\,))(x_1) = -u_1(\lambda,x_2)u_2(\lambda,x_1),\\
&0 \neq c = W(u_1(\lambda,\,\cdot\,),u_2(\lambda,\,\cdot\,))(x_2) = u_1(\lambda,x_2)u_2(\lambda,x_1).
\end{split}
\end{equation}
The identities in \eqref{3.8} immediately imply
\begin{equation}
0 \neq 2c = u_1(\lambda,x_2)u_2(\lambda,x_1) -u_1(\lambda,x_2)u_2(\lambda,x_1) = 0,
\end{equation}
which is an obvious contradiction.

Finally, to prove that $(iv)$ implies $(i)$, suppose that $u_1(\lambda,\,\cdot\,)>0$ is a strictly positive solution of $(\tau - \lambda)u = 0$ on $(a,b)$.  Fixing a point $x_0\in (a,b)$, Lemma \ref{l4.1}
 applied to $(c,d)=(a,b)$ implies that all real-valued solutions of $(\tau - \lambda)u = 0$ on $(a,b)$ are given by
\begin{equation}\lb{3.10}
u(\lambda,x) = u_1(\lambda,x)\bigg[c_1 + c_2\int_{x_0}^x\frac{dt}{p(t)u_1(\lambda,t)^2}\bigg],\quad c_j\in \bbR,\, j\in \{1,2\},\, x\in (a,b).
\end{equation}
The function $\int_{x_0}^xdt\, \big[p(t)u_1(\lambda,t)^2\big]^{-1}$ is strictly increasing on $(a,b)$, so the factor in square brackets in \eqref{3.10} has at most one zero in $(a,b)$.  Since $u_1(\lambda,\,\cdot\,)$ does not vanish in $(a,b)$, it follows from \eqref{3.10} that every real-valued solution of $(\tau - \lambda)u = 0$ on $(a,b)$ has at most one zero in $(a,b)$.  Thus, $\tau - \lambda$ is disconjugate on $(a,b)$.
\end{proof}
%%%%%%%%%%%

In the following theorem, the principal result of this paper, we show that $\lambda_0\in \bbR$ is a lower bound of $T_F$ (equivalently, a lower bound of $T_{min}$) if and only if $(\tau - \lambda_0)u = 0$ has a strictly positive solution on $(a,b)$.  In turn, this is equivalent to $\tau - \lambda_0$ on $(a,b)$ being disconjugate on $(a,b)$.

%%%%%%%%%%%
\begin{theorem}\lb{t4.3} 
Assume Hypothesis \ref{h2.1} and let $\lambda_0\in \bbR$.  Then the following statements $(i)$--$(iv)$ are equivalent. \\[1mm]
$(i)$  $T_{min}\geq \lambda_0I_{L^2((a,b);r\,dx)}$.\\[1mm]
$(ii)$ $T_F\geq \lambda_0I_{L^2((a,b);r\,dx)}$.\\[1mm]
$(iii)$  $\tau-\lambda_0$ is disconjugate on $(a,b)$.\\[1mm]
$(iv)$  There exists a strictly positive solution $u_0(\lambda_0,\,\cdot\,)>0$ of $(\tau - \lambda_0)u = 0$ on $(a,b)$.
\end{theorem}
%%%%%%%%%%%
\begin{proof}
The equivalence of statements $(i)$ and $(ii)$ follows immediately from the abstract fact that a densely defined lower semibounded operator and its Friedrichs extension share the same greatest lower bound (cf., e.g., \cite[Theorem 10.17 $(i)$]{Sc12}).  Theorem \ref{t4.2}
 immediately implies the equivalence of statements $(iii)$ and $(iv)$.  To complete the proof, we show that $(iv)$ implies $(i)$ and that $(i)$ implies $(iii)$.

To prove that $(iv)$ implies $(i)$, let $u_0(\lambda_0,\,\cdot\,)>0$ be a strictly positive solution of $(\tau - \lambda_0)u = 0$ on $(a,b)$.  In order to show $(i)$, it suffices to prove $\dot T\geq \lambda_0I_{L^2((a,b);r\,dx)}$, since $T_{min}$ is the closure of $\dot T$.  Let $f\in \dom\big(\dot T\big)$.  Since $(\tau - \lambda_0)u_0(\lambda_0,\,\cdot\,)=0$ a.e.~on $(a,b)$ and $u_0(\lambda_0,\,\cdot\,)$ does not vanish on $(a,b)$, $q$ may be recovered pointwise a.e.~on $(a,b)$ by
\begin{equation}\lb{4.11}
q = \lambda_0r - s\frac{u_0^{[1]}(\lambda,\,\cdot\,)}{u_0(\lambda,\,\cdot\,)} + \frac{\big(u_0^{[1]}(\lambda,\,\cdot\,)\big)'}{u_0(\lambda,\,\cdot\,)}\, \text{ a.e.~on $\in(a,b)$}.
\end{equation}
In addition, we recall the following version of Jacobi's factorization identity (cf.~\cite[Equation (11.66)]{EGNT13}):
\begin{align}
&-\big(f^{[1]} \big)'+s f^{[1]}+\frac{(u_0^{[1]}(\lambda_0,\,\cdot\,))'}{u_0(\lambda_0,\,\cdot\,)}f-s\frac{u_0^{[1]}(\lambda_0,\,\cdot\,)}{u_0(\lambda_0,\,\cdot\,)}f\lb{4.12a}\\
&\quad =-\frac{1}{u_0(\lambda_0,\,\cdot\,)}\bigg[pu_0(\lambda_0,\,\cdot\,)^2\bigg(\frac{f}{u_0(\lambda_0,\,\cdot\,)} \bigg)'\bigg]'  \, \text{ a.e.~in $(a,b)$},\no
\end{align}
Applying \eqref{4.12a}, \eqref{4.11}, and the fact that $f$ has compact support in $(a,b)$ to integrate by parts,
\begin{equation}\lb{4.13}
\begin{split}
&\big(f,\dot Tf\big)_{L^2((a,b);r\,dx)}\\
&\quad= \lambda_0 \|f\|_{L^2((a,b);r\,dx)}^2 + \int_a^b dx\, p(x)u_0(\lambda_0,x)^2 
\bigg|\bigg(\f{f(x)}{u_0(\lambda_0,x)}\bigg)'\bigg|^2.
\end{split}
\end{equation}
The integral on the right-hand side in \eqref{4.13} is nonnegative.  Since $f\in\dom\big(\dot T\big)$ was arbitrary, it follows that $\dot T\geq \lambda_0 I_{L^2((a,b);r\,dx)}$.

To prove $(i)$ implies $(iii)$ we argue by way of contradiction.  Assume that $(i)$ holds and $\tau - \lambda_0$ is not disconjugate on $(a,b)$.  Statement $(i)$ implies
\begin{equation}\lb{4.14a}
T_F\geq \lambda_0I_{L^2((a,b);r\,dx)},
\end{equation}
and since $\tau - \lambda_0$ is not disconjugate on $(a,b)$, there exists a nontrivial real-valued solution $u_1(\lambda_0,\,\cdot\,)$ of $(\tau - \lambda_0)u = 0$ on $(a,b)$ with two consecutive zeros $x_1,x_2\in (a,b)$, $x_1<x_2$.  Without loss of generality, we may assume that $u_1(\lambda_0,\,\cdot\,)>0$ on $(x_1,x_2)$.  Let $T_{F,(x_1,x_2)}$ denote the Friedrichs extension of the minimal operator $T_{min,(x_1,x_2)}$ corresponding to the restriction of $\tau$ to the interval $(x_1,x_2)$, that is, $\tau|_{(x_1,x_2)}$, and let $\lambda_1(T_{F,(x_1,x_2)})$ denote the principal eigenvalue of $T_{F,(x_1,x_2)}$.  One notes that the differential expression $\tau|_{(x_1,x_2)}$ is regular on $(x_1,x_2)$.  Invoking \eqref{2.14}, a calculation analogous to \eqref{4.13} yields
\begin{align}
Q_{F,(x_1,x_2)}(f,f) &= \lambda_0\|f\|_{L^2((x_1,x_2);r\,dx)}^2 + \int_{x_1}^{x_2} dx\, p(x)u_0(\lambda_0,x)^2\bigg|\bigg(\f{f(x)}{u_0(\lambda_0,x)}\bigg)'\bigg|^2\no\\
&\geq \lambda_0\|f\|_{L^2((x_1,x_2);r\,dx)}^2,\quad f\in \dom\big(|T_{F,(x_1,x_2)}|^{1/2}\big)\backslash\{0\}.\lb{4.14}
\end{align}
Note that equality holds in \eqref{4.14} if $f=u_0(\lambda_0,\,\cdot\,)|_{(x_1,x_2)}$.  Therefore, it follows from Theorem \ref{t3.2} and the simplicity of the eigenvalues of $T_{F,(x_1,x_2)}$ that $\lambda_0=\lambda_1(T_{F,(x_1,x_2)})$ and $u_0(\lambda_0,\,\cdot\,)|_{(x_1,x_2)}$ is a corresponding strictly positive eigenfunction which is unique up to constant multiples.  Choose $x_3\in (a,x_1)$ and $x_4\in (x_2,b)$ such that the following chain of inclusions holds
\begin{equation}\lb{4.15}
(x_1,x_2)\subset (x_3,x_4) \subset (a,b),
\end{equation}
and let $T_{F,(x_3,x_4)}$ in $L^2((x_3,x_4);r\,dx)$ denote the Friedrichs extension of the minimal operator $T_{min,(x_3,x_4)}$ corresponding to $\tau|_{(x_3,x_4)}$. Once more, $\tau|_{(x_3,x_4)}$ is regular on $(x_3,x_4)$.  By \eqref{4.14a} and the variational characterization of the principal eigenvalue $\lambda_1(T_F)$ of $T_F$ in Theorem \ref{t3.1},
\begin{align}
\lambda_0\leq \lambda_1(T_F) &= \min_{0\neq f\in \dom(|T_F|^{1/2})}R_{T_F}\big(f;L^2((a,b);r\,dx)\big)\lb{4.16}\\
&\leq \min_{0\neq f\in \dom(|T_{F,(x_3,x_4)}|^{1/2})}R_{T_{F,(x_3,x_4)}}\big(f;L^2((x_3,x_4);r\,dx)\big)\no\\
&< \min_{0\neq f\in \dom(|T_{F,(x_1,x_2)}|^{1/2})}R_{T_{F,(x_1,x_2)}}\big(f;L^2((x_1,x_2);r\,dx)\big)\no\\
&=\lambda_1(T_{F,(x_1,x_2)}) = \lambda_0,     \no
\end{align}
which is a contradiction.  Note that the inequality ``$\leq$'' in \eqref{4.16} follows by extending $f\in \dom\big(|T_{F,(x_3,x_4)}|^{1/2}\big)\backslash\{0\}$ to $\widetilde{f}\in\dom\big(|T_{F,(a,b)}|^{1/2}\big)\backslash\{0\}$ by setting $\widetilde{f}$ equal to zero on $(a,b)\backslash (x_3,x_4)$; that is,
\begin{equation}
\widetilde{f}(x) =
\begin{cases}
f(x),& x\in (x_3,x_4),\\
0,& x\in (a,b)\backslash(x_3,x_4),
\end{cases}
\end{equation}
thereby ensuring $\widetilde{f}\in\dom\big(|T_{F,(a,b)}|^{1/2}\big)$.  The strict inequality ``$<$'' in \eqref{4.16} follows from \eqref{4.15} and the strict domain monotonicity of the principal eigenvalue of the Friedrichs extension shown in Theorem \ref{t3.4}.
\end{proof}
%%%%%%%%%%%

%%%%%%%%%%%
\begin{remark} \lb{r4.4} 
$(i)$ A simple algebraic manipulation shows that the integrand in \eqref{4.13} coincides a.e.~on $(a,b)$ with
\begin{equation}
p^{-1}u_0(\lambda_0,\,\cdot\,)^2\bigg|\frac{f^{[1]}u_0(\lambda_0,\,\cdot\,)-fu_0^{[1]}(\lambda_0,\,\cdot\,)}{u_0(\lambda_0,\,\cdot\,)^2}\bigg|^2.
\end{equation}
Hence, since $f$ is compactly supported in $(a,b)$, the fact that the integral in \eqref{4.13} is finite is consistent with the assumption that $p^{-1}\in L^1_{loc}((a,b);dx)$. \\[1mm] 
$(ii)$ Special cases of Theorem \ref{t3.4} have a long history, and several variants \dash under varying (stronger) hypotheses on the coefficients in $\tau$ \dash can be found in the literature. We refer, for instance, to \cite[Sect.~XI.6, p.~401--402]{Ha02}. 
\hfill$\diamond$
\end{remark}
%%%%%%%%%%%

We conclude by recalling an intimate connection between nonoscillatory behavior and the limit point case for $\tau$ at an endpoint. This result is well-known within the context of traditional three-term Sturm--Liouville differential expressions of the form $\tau u = r^{-1}[-(pu')'+qu]$. It was first derived by Hartman \cite{Ha48} in the particular case $r=1$ in 1948.  Three years later, in 1951, Rellich \cite{Re51} extended the result to the general three-term case under some additional smoothness assumptions on $p, q, r$; for additional results in this direction see \cite[Lemma~C.1]{Ge93}, \cite[Sect.~5]{LM03}, \cite[Corollary~3.2]{NZ92}. 

%%%%%%%%%%%
\begin{theorem}[see {\cite[Theorem~11.7, Corollary~11.8]{EGNT13}}] \ \lb{t4.5}
Assume Hypothesis \ref{h2.1}, let $c \in (a,b)$,  and suppose that $\tau-\lambda$ is nonoscillatory at $b$ for some 
$\lambda \in \bbR$.  If 
\begin{equation} 
\int_c^b |r(x)/p(x)|^{1/2} dx=\infty,    \lb{4.20}
\end{equation} 
then $\tau$ is in the limit point case at $b$.  An analogous result holds at $a$, in particular, if $\tau-\lambda_a$ is nonoscillatory at $a$ for some 
$\lambda_a\in \bbR$ and $\tau-\lambda_b$ is nonoscillatory at $b$ for some $\lambda_b \in \bbR$, and 
\begin{equation} \lb{4.21}
\int_a^c dx \, |r(x)/p(x)|^{1/2} = \infty, \quad \int_c^b dx \, |r(x)/p(x)|^{1/2} = \infty,
\end{equation}
then $\dot T_{min}$ is essentially self-adjoint and hence $T_{min}=T_{max}$ is self-adjoint.
\end{theorem}
%%%%%%%%%%%
\begin{proof} For convenience of the reader we present Hartman's short argument as used in 
\cite[Theorem~11.7]{EGNT13}. It suffices to treat the case where $\tau-\lambda$ is nonoscillatory at $b$. Then by Theorem \ref{t2.8} there exists a principal solution, say $u_b(\la,\dott)$, of $\tau u = \lambda u$.  If $x_0$ strictly exceeds the largest zero of $u_b(\la,\dott)$ in $(c,b)$, then $\hatt u_b(\la,\dott)$ defined by
\begin{equation}\lb{4.22}
\hatt u_b(\la,x)=u_b(\la,x)\int_{x_0}^x \frac{dx'}{p(x')u_b(\la,x')^2}, \quad x\in (x_0,b),
\end{equation}
is a nonprincipal solution on $(x_0,b)$ by Theorem \ref{t2.10}, and consequently,
\begin{equation}\lb{4.23}
\int_{x_0}^b \frac{dx}{|p(x)|\hatt u_b(\la,x)^2}<\infty.
\end{equation}
Arguing by contradiction, and hence assuming $\tau$ to be in the limit circle case at $b$, one concludes that 
\begin{equation}\lb{4.24}
\int_{x_0}^b r(x) dx \, \hatt u_b(\la,x)^2 <\infty.
\end{equation}
Consequently, H\"older's inequality yields the contradiction, 
\begin{align}\lb{4.25}
\int_{x_0}^b dx \, |r(x)/p(x)|^{1/2} \leq \bigg|\int_{x_0}^b  r(x) dx \, \hatt u_b(\la,x)^2 \bigg|^{1/2}
\bigg|\int_{x_0}^b \frac{dx}{|p(x)|\hatt u_b(\la,x)^2}\bigg|^{1/2}<\infty.
\end{align}
\end{proof}
%%%%%%%%%%%

This result extends to certain situations where $p$ changes sign but has a constant sign near $a$ and $b$, see \cite{EGNT13}.

\medskip

%%%%%%%%%%%%%%%%%%%%%%%%%%%%%%%%%%%%%
\noindent 
{\bf Acknowledgments.} We are indebted to Aleksey Kostenko for helpful comments. 
%%%%%%%%%%%%%%%%%%%%%%%%%%%%%%%%%%%%%

\medskip

%%%%%%%%%%%%%%%%%%%%%%%%%%%%%%%%%%%%%
\noindent 
{\bf Declarations.} The authors state that there is no conflict of interest. \\
No datasets were generated or analyzed during the current study.
%%%%%%%%%%%%%%%%%%%%%%%%%%%%%%%%%%%%%

%%%%%%%%%%%%%%%%%%%%%%%%%%%%%%%%
%%%%%%%%%%%%%%%%%%%%%%%%%%%%%%%%
 
\end{document}